\documentclass[12pt]{article}
\usepackage{amssymb,amsmath,latexsym}
\newtheorem{thm}{Theorem}[section]

\newtheorem{prop}[thm]{Proposition}

\newtheorem{rem}[thm]{Remark}
\newtheorem{dfn}[thm]{Definition}
\newtheorem{ex}[thm]{Example}
\numberwithin{equation}{section}
\newenvironment{proof}{\trivlist \item[\hskip
	\labelsep{{\it\underline{Proof}.\/}}]}{\ \rule{0.5em}{0.5em}}%
{\endtrivlist}
{\endtrivlist}
\begin{document}
\title{\bf $3$-Leibniz bialgebras ($3$-Lie bialgebras)}
\author{A. Rezaei-Aghdam$^{a}$\thanks{e-mail: rezaei-a@azaruniv.edu}\,\,\, L. Sedghi-Ghadim $^{b}$\thanks{e-mail: sedghi88@azaruniv.edu}\,\,\\ \\
	{\small {$^a${\em Department of Physics, Azarbaijan Shahid Madani
				University , 51745-406, Tabriz, Iran.}}}\\
	{\small {$^b${\em Department of Mathematics, Azarbaijan Shahid Madani
				University , 51745-406, Tabriz, Iran.}}}}
\maketitle
\hspace{13.5pt}
\begin{abstract}
The aim of this paper is to extend the notion of bialgebra for Leibniz algebras (and Lie algebras) to $3$-Leibniz algebras (and $3$-Lie algebras) by use of the cohomology complex of $3$-Leibniz algebras. Also, some theorems about Leibniz bialgebras are extended and proved in the case of $3$-Leibniz bialgebras ($3$-Lie bialgebras). Moreover, a new theorem on the
correspondence between $3$-Leibniz bialgebra and its associated Leibniz bialgebra is proved. Finally, some examples are discussed in detail.
\end{abstract}
\section{\bf Introduction}
\hspace{13.5pt}
From historical point of view Kurosh introduced the notion of multilinear operator algebra for the first time in Refs\cite{Kur1,Kur2}. However, for these algebras
one of the most important consequences of Jacobi identity is overlooked (i.e., the derivation property of $\mathrm{ad}_x$ for an element $x$ of the algebra). Later in \cite{Filippov} Filippov introduced the $n$-Lie algebra which preserves main properties of Jacobi identity. In Ref\cite{Kasymov} the $n$-Lie modules and representation of $n$-Lie algebras, generalization of Engel's and Lie's theorems and also Cartan's criterion for solvability of $n$-Lie algebra have been studied by Kasymov. In the past two decades the study of the $n$-Leibniz algebra, its cohomology \cite{Rotk}, their classifications \cite{Classify, Bai1} and deformation of $n$-Leibniz algebras  are under investigation (see for instance \cite{Azcar1} for a review see \cite{Azcar2}). Recently, the application of $3$-Lie algebra in the M theory \cite{Bag1,Bag2} has led this branch of mathematics to receive the most attention among physicists \cite{Azcar2}. One of the most applicable objects in mathematical physics, especially in integrable systems is the Lie bialgebra. In this manner, the generalization of the concept of Lie bialgebra to the $n$-Lie bialgebra (in general) and especially $3$-Lie bialgebra is a good problem from the abstract point of view. Indeed, there are some attempts in this direction from the co-algebra point of view (see \cite{Morad} and \cite{Bai2}). Here we will study these facts by use of the cohomology of $n$-Leibniz algebra \cite{Rotk} for $3$-Leibniz algebra in general and then for $3$-Lie algebra in particular\footnote{Note that the Lie algebra is a special case of Leibniz algebra \cite{Lod2}.}. The outline of the paper is as follows.
\newline
In Section $2$ for self containing of the paper, we review the basic definitions and theorems on $n$-Leibniz algebra \cite{Lod1}, $n$-Lie algebra, its associated Leibniz algebra \cite{Gaut} and Leibniz bialgebra \cite{Reza1}.
\newline
In Section $3$ after the separation of the first, the second and the third $3$-Leibniz algebra we give the definition of $3$-Leibniz bialgebra $(\mathcal{A},\gamma)$ for the different $i$-th $3$-Leibniz algebra $\mathcal{A}$. Then, as a proposition we show that the dual space $\mathcal{A}^{\ast}$ with $\mu^t$ is a $3$-Leibiz bialgebra. The investigation of $3$-Leibniz bialgebra $(\mathcal{A},\mathcal{A}^\ast)$ in terms of structure constants of $3$-Leibniz algebras $\mathcal{A}$ and $\mathcal{A}^\ast$ are given in Section $4$; and at the end of this section some examples of $3$-Leibniz bialgebras are obtained by using matrix calculations. In Section $5$ the correspondence between $3$-Leibniz bialgebra and its associated Leibniz bialgebra is given as a theorem.  The definition of $3$-Lie bialgebra as an especial case and the reformulation of this definition in terms of structure constants of $3$-Lie algebras $\mathcal{A}$ and $\mathcal{A}^\ast$ are provided in Section $6$. The matrix form of this reformulation is applied for the calculation of some low dimensional $3$-Lie bialgebras at the end of Section 6. Some concluding remarks are given in Section $7$.

\section{\bf Basic definitions and theorems}
For self containing of the paper, let us recall some basic
definitions and theorems about $n$-Leibniz algebras, $n$-Lie algebras and also Leibniz bialgebras \cite{Reza1}.
\begin{dfn}\cite{Lod1}
 A vector space $\mathcal{A}$ equipped with  an $n$-linear operation $[., . . ., .]: \mathcal{A}^{\otimes n}\longrightarrow\mathcal{A}$, such that for all $x_{1}, . . ., x_{n-1}\in\mathcal{A}$ the map $\mathrm{ad}_{(x_{1}, . . ., x_{n-1})}:\mathcal{A}\longrightarrow \mathcal{A}$ which is given by
\begin{align}
\mathrm{ad}_{(x_{1}, . . ., x_{n-1})}(x):=[x,x_{1}, . . ., x_{n-1}],\label{adequation1}
\end{align}
is called an $n$-Leibniz algebra if the map
$\mathrm{ad}_{(x_{1}, . . ., x_{n-1})}:\mathcal{A}\longrightarrow \mathcal{A}$
 be a derivation with respect to $[., . . ., .]$ i.e.
{\small
\begin{align}
[[y_{1}, . . ., y_{n}],x_{1}, . . ., x_{n-1}]=\sum_{i=1}^{n}[y_{1}, . . ., y_{i-1},[y_{i},x_{1}, . . ., x_{n-1}],y_{i+1}, . . ., y_{n}],
\label{fundamental identity1}
\end{align}
}
this identity is called fundamental identity for $n$-Leibniz algebra.
\end{dfn}
\begin{dfn}\cite{Lod1}
A representation of the $n$-Leibniz algebra $\mathcal{A}$  is a vector space $M$ equipped with $n$ actions of
\begin{align*}
\rho_{j}:\mathcal{A}^{\otimes(j-1)}\otimes M\otimes \mathcal{A}^{\otimes(n-j)}\longrightarrow M\qquad j=1,2, . . ., n,
\end{align*}
satisfying $2n-1$ equations, which are obtained from \eqref{fundamental identity1} by letting exactly one of the variables $x_{1}, . . ., x_{n-1},y_{1}, . . ., y_{n}$
be in $M$ and all the others in $\mathcal{A}$. In the other word, $M$ is an $n$-Leibniz module. The notion of representation of an $n$-Leibniz algebra for $n=2$ coincides with the corresponding notion representation of Leibniz algebra in \cite{Lod2}.
\end{dfn}
\begin{thm}\cite{Gaut}
Let $\mathcal{A}$ be an $n$-Leibniz algebra and set $\mathfrak{g}:=\mathcal{A}^{\otimes (n-1)}$ then there is a Leibniz algebra structure on the space
$\mathfrak{g}$ with the following bracket:
{\small
\begin{align}
[x_{1}\otimes . . . \otimes x_{n-1}, y_{1}\otimes . . . \otimes y_{n}]=\sum_{i=1}^{n-1}y_{1}\otimes . . . \otimes y_{i-1}\otimes [x_{1}, . . ., x_{n-1}, y_{i}]
\otimes y_{i+1}\otimes . . . \otimes y_{n-1}.
\end{align}
}
$\mathfrak{g}$ with the above bracket is called associated Leibniz algebra.
\end{thm}
\begin{dfn}\cite{ Rotk, Gaut}
Let $\mathcal{A}$ be an $n$-Leibniz algebra and $\mathfrak{g}:=\mathcal{A}^{\otimes (n-1)}$ be its associated Leibniz algebra. The $p$-cochain of $\mathcal{A}$
$(p\geq 1)$ with coefficients in $\mathcal{A}$ is a linear map from $\mathfrak{g}^{\otimes (p-1)}\otimes \mathcal{A}$ to $\mathcal{A}$.
Setting $\Gamma L^{0}(\mathcal{A},\mathcal{A}):=\mathfrak{g}$ for the space of $0$-cochains and $\Gamma L^{p}(\mathcal{A},\mathcal{A})$ for the space of $p$-cochains
the coboundary map is given  by\cite{Rotk}
\begin{align*}
&d^{p}:\Gamma L^{p}(\mathcal{A},\mathcal{A})\longrightarrow \Gamma L^{p+1}(\mathcal{A},\mathcal{A})\\
&(d^{0}(x_{1}\otimes . . . \otimes x_{n-1}))(x)=-[x_{1},  . . ., x_{n-1}, x],
\end{align*}
{\small
\begin{align}
\begin{split}
(d^{p}(\alpha)(X_{1}, . . ., X_{p-1}, Y)
&=\sum_{i=1}^{p-2}\sum_{j=i+1}^{p-1}(-1)^{i}\alpha(X_{1}, . . ., \widehat{X_{i}}, . . ., X_{j-1},[X_{i}, X_{j}], X_{j+1}, . . ., X_{p-1}, Y)\\
& \;\;+\sum_{i=1}^{p-1}(-1)^{i}\alpha(X_{1}, . . ., \widehat{X_{i}}, . . ., X_{p-1},\{X_{i}, Y\})\\
&\;\; +(-1)^{p}\alpha(X_{1}, . . ., X_{p-1}, [y_{1},
 . . . , y_{n}])\\
&\;\;+\sum_{i=1}^{p-1}(-1)^{i+1}\{X_{i}, \alpha(X_{1}, . . ., \widehat{X_{i}}, . . ., X_{p-1}, Y)\}\\
&\;\; +(-1)^{p+1}\sum_{i=1}^{n}[y_{1}, . . ., y_{i-1},\alpha(X_{1}, . . . , X_{p-1}, y_{i}), . . ., y_{n}],
\end{split}
\end{align}
}
where $\alpha\in\Gamma L^{p}(\mathcal{A},\mathcal{A})$,   $X_{i}\in \mathfrak{g}$\;  for  $i=1, . . ., p-1$,  $Y=y_{1}\otimes . . . \otimes y_{n}\in\mathcal{A}^{\otimes n}$ and for
$X=x_{1}\otimes . . . \otimes x_{n-1}$  set $\{X,Y\}:=\sum_{i=1}^{n}y_{1}\otimes . . . \otimes y_{i-1}\otimes[x_{1}, . . ., x_{n-1},y_{i}]\otimes . . . \otimes y_{n}$.
\end{dfn}
\begin{dfn}\cite{Reza1} A  Leibniz bialgebra $(\mathfrak{g},\delta)$ is a (right or left) Leibniz algebra $\mathfrak{g}$ with a linear map (cocommutator) $\delta:\mathfrak{g}\longrightarrow \mathfrak{g}\otimes \mathfrak{g}$ such that
\begin{itemize}
\item
$\delta$ is a $1$-cocycle on $\mathfrak{g}$ with values in $\mathfrak{g}\otimes \mathfrak{g}$
\begin{align}
[X,\delta(Y)]_{L}+[\delta(X),Y]_{R}-\delta([X,Y])=0,\label{1-cocycleforL}
\end{align}
where $[.,.]_{L}$ and $[.,.]_{R}$ represent the left and  the right action of  $\mathfrak{g}$ on $\mathfrak{g}\otimes \mathfrak{g}$ respectively such that $\mathfrak{g}\otimes \mathfrak{g}$ becomes a $\mathfrak{g}$-module.
\item
$\delta ^{t}:\mathfrak{g}^{\ast}\otimes \mathfrak{g}^{\ast}\longrightarrow \mathfrak{g}^{\ast}$ defines a Leibniz bracket on $\mathfrak{g}^{\ast}$.
With the notation $[\xi,\eta]_{\ast}=\delta^{t}(\xi\otimes \eta)$, for any $\xi,\eta\in\mathfrak{g}^{\ast}$ and $X\in\mathfrak{g}$ we will have
\begin{align}
\langle [\xi,\eta]_{\ast},X\rangle=\langle\delta^{t}(\xi\otimes \eta),X\rangle=\langle \xi\otimes \eta,\delta(X),\rangle,
\end{align}
where $\langle ,\rangle$ is the natural pairing between $\mathfrak{g}$ and $\mathfrak{g}^{\ast}$.
\end{itemize}
\end{dfn}
Note that with respect to the type of the Leibniz algebra $\mathfrak{g}$ and also its actions on the $\mathfrak{g}\otimes \mathfrak{g}$; the $1$-cocycle condition \eqref{1-cocycleforL} can be rewritten in the following forms:
{\small
\begin{align}
&\delta([X,Y])
=(\mathrm{ad}_{X}^{(l)}\otimes 1)(\delta(Y))+(\mathrm{ad}_{Y}^{(r)}\otimes 1)(\delta(X)),\label{1-cocycle1lorr-r}\\
&\delta([X,Y])
=(1\otimes \mathrm{ad}_{Y}^{(r)}+\mathrm{ad}_{Y}^{(r)}\otimes 1)(\delta(X)),\label{1-cocycler-lorr}\\
& \delta([X,Y])
=(1\otimes \mathrm{ad}_{X}^{(l)}+ \mathrm{ad}_{X}^{(l)}\otimes 1)(\delta(Y)),\label{1-cocyclel-lorr}\\
&\delta([X,Y])
=(1\otimes \mathrm{ad}_{X}^{(l)})(\delta(Y))+(1\otimes \mathrm{ad}_{Y}^{(r)})(\delta(X)),\label{1-cocyclelorr-l}
\end{align}
}
where for the cases \eqref{1-cocycle1lorr-r} and \eqref{1-cocyclelorr-l} $\mathfrak{g}$ can be a left or  a right Leibniz algebra and in the case \eqref{1-cocycler-lorr} (\eqref{1-cocyclel-lorr}) $\mathfrak{g}$ is a right (left) Leibniz algebra.
\begin{dfn}\cite{Filippov} An $n$-Lie algebra $(\mathcal{A},[., . . ., .])$ is a vector
space over a field $\mathbb{F}$ together with a skew-symmetric $n$-linear map
$[., . . ., .]:\mathcal{A}^{\otimes n}\longrightarrow \mathcal{A}$ such that
\begin{align*}
[x_1, . . ., x_{n-1},[y_1, . . ., y_n]]=\sum_{i=1}^{n}[y_1, . . ., [x_1, . . ., x_{n-1},y_i],y_{i+1}, . . ., y_n],
\end{align*}
for all $x_1, . . ., x_{n-1},y_1, . . ., y_n\in\mathcal{A}$. This condition is called the fundamental identity or the Filippov identity.
\end{dfn}
\begin{dfn}\cite{Kasymov}\label{rep3l}
If $\mathcal{A}$ and $V$ be an $n$-Lie algebra  and a vector space over a field  $\mathbb{F}$ respectively, then a polylinear mapping $\rho:\mathcal{A}^{\otimes n-1}\longrightarrow End(V)$	will be said to be a representation of $\mathcal{A}$ in $V$ if the operators
$\rho(x_1, . . ., x_{n-1}),\forall x_i\in\mathcal{A}$ be skew-symmetric functions of their arguments and satisfy in the following identities:
\begin{align}
&[\rho(x_1, . . ., x_{n-1}), \rho(y_1, . . ., y_{n-1})]=\sum_{i=1}^{n-1}\rho(y_1, . . ., y_{i-1},[x_1, . . ., x_{n-1},y_i],y_{i+1}, . . ., y_{n-1}),\label{c1}\\
&\rho(x_1, . . ., x_{n-2},[y_1, . . ., y_n])=\sum_{i=1}^{n}(-1)^{i+1}\rho(y_1, . . ., \hat{y_i}, . . ., y_n)\rho(x_1, . . ., x_{n-2},y_i),\label{c2}
\end{align}
where $x_1, . . ., x_{n-1},y_1, . . ., y_{n-1}\in \mathcal{A}$. In this case, $V$ is said to be an ($n$-Lie) $\mathcal{A}$-module. For $n=3$ we have
\begin{align}
&[\rho(x_1, x_2),\rho(y_1,y_2)]=\rho([x_1,  x_2, y_1], y_2)+\rho(y_1, [x_1, x_2, y_2]),\\
&\rho(x_1, [y_1, y_2, y_3])=\rho(y_2, y_3)\rho(x_1, y_2)-\rho(y_1, y_3)\rho(x_1, y_2)+\rho(y_1, y_2)\rho(x_1, y_3).
\end{align}	
\end{dfn}
\begin{dfn}\cite{Gaut}\label{cohol}
Let $\mathcal{A}$ be a $3$-Lie algebra, an $\mathcal{A}$-valued $p$-cochain is a linear map \linebreak $\psi:(\mathcal{A}\otimes \mathcal{A})^{\otimes(p-1)}\otimes \mathcal{A}\longrightarrow \mathcal{A}$ such that the coboundary operator is given by:
\begin{align*}
d^{p}\psi(x_1, . . ., x_{2p+1})=&\sum_{j=1}^{p}\sum_{k=2j+1}^{2p+1}(-1)^{j}\psi(x_1, . . ., \hat{x}_{j-1},\hat{x}_j, . . ., [x_{2j-1},x_{2j},x_k], . . ., x_{2p+1})\\
&+\sum_{k=1}^{p}[x_{2k-1},x_{2k},\psi(x_1, . . ., \hat{x}_{2k-1},\hat{x}_{2k}, . . ., x_{2p+1})]\\
&+(-1)^{p+1}[x_{2p-1},\psi(x_1, . . ., x_{2p-2},x_{2p}),x_{2p+1}]\\
&+(-1)^{p+1}[\psi(x_1, . . ., x_{2p-1}),x_{2p},x_{2p+1}]\nonumber
\end{align*}
\end{dfn}
\section{\bf  $3$-Leibniz bialgebras}
\hspace{13.5pt}
Since the bracket $[., . . ., .]$ for the $n$-Leibniz algebra is not antisymmetric hence we define the map $\mathrm{ad}_{(x_{1}, . . ., \widehat{x_{i}}, . . ., x_{n})}:\mathcal{A}\longrightarrow \mathcal{A}$, for all $x_{1}, . . ., x_{n}$ in $\mathcal{A}$
 as follows:
\begin{align}
\mathrm{ad}_{(x_{1}, . . ., \widehat{x_{i}}, . . ., x_{n})}(x):=[x_{1}, . . ., x_{i-1},x,x_{i+1}, . . ., x_{n}],\quad\mbox{for $i=1, . . .,  n$}.\label{cadequation}
\end{align}
\begin{dfn}
An {\textit{$i$-th $n$-Leibniz algebra}}, is a vector space $\mathcal{A}$ equipped with an $n$-linear operation $[., . . ., .]: \mathcal{A}^{\otimes n}\longrightarrow\mathcal{A}$ such that the map $\mathrm{ad}_{(x_{1}, . . ., \widehat{x_{i}}, . . ., x_{n})}$ is a derivation with respect to $[., . . ., .]$  i.e.
{\footnotesize
\begin{align}
[x_{1}, . . ., x_{i-1},[y_{1}, . . .,  y_{n}],x_{i+1}, . . .  ,x_{n}]=\sum_{j=1}^{n}[y_{1}, . . ., y_{j-1},[x_{1}, . . ., x_{i-1},y_{j},x_{i+1}, . . ., x_{n}]
,y_{j+1}, . . ., y_{n}],\label{nli}
\end{align}
}
therefore, for any $i$ we have  an  $n$-Leibniz algebra.
\end{dfn}
\begin{rem}
In \cite{Lod1} and \cite{Rotk} the map $\mathrm{ad}_{(x_{1}, . . ., \widehat{x_{i}}, . . ., x_{n})}$  is considered as a derivation with respect to $[., . . ., .]$ for the cases $i=n$ and $i=1$. In the other words, the $i$-th $n$-Leibniz algebras for $i=2, . . ., n-1$ have not been considered.
\end{rem}
For $n=3$ we have three types of $3$-Leibniz identities
\begin{itemize}
\item
 If the map $\mathrm{ad}_{(\widehat{x_{1}},x_{2},x_{3})}$ be a derivation with respect to $[.,.,.]$ then we will have the first $3$-Leibniz identity as follows:
{\small
\begin{align}
[[y_{1}, y_{2}, y_{3}], x_{2}, x_{3}]=[[y_{1}, x_{2}, x_{3}], y_{2}, y_{3}]+[y_{1}, [y_{2}, x_{2}, x_{3}], y_{3}]+[y_{1}, y_{2}, [y_{3}, x_{2}, x_{3}]]. \label{f3li}
\end{align}
}
\item
If the map $\mathrm{ad}_{(x_{1},\widehat{x_{2}},x_{3})}$ be a derivation with respect to $[.,.,.]$ then we will have the second $3$-Leibniz identity as follows:
{\small
\begin{align}
[x_{1},[y_{1},y_{2},y_{3}],x_{3}]=[[x_{1},y_{1},x_{3}],y_{2},y_{3}]+[y_{1},[x_{1},y_{2},x_{3}],y_{3}]+[y_{1},y_{2},[x_{1},y_{3},x_{3}]].\label{s3li}
\end{align}
}
\item
If the map $\mathrm{ad}_{(x_{1},x_{2},\widehat{x_{3}})}$ be a derivation with respect to $[.,.,.]$ then we will have the third $3$-Leibniz identity as follows:
{\small
\begin{align}
[x_{1},x_{2},[y_{1},y_{2},y_{3}]]=[[x_{1},x_{2},y_{1}],y_{2},y_{3}]+[y_{1},[x_{1},x_{2},y_{2}],y_{3}]+[y_{1},y_{2},[x_{1},x_{2},y_{3}]].\label{t3li}
\end{align}
}
\end{itemize}
Before defining the $3$-Leibniz bialgebra let us define special actions such that $\mathcal{A}^{\otimes 3}$ be a $3$-Leibniz module.
We define the following cases of actions for any $x_1, x_2, x_3, y_1, y_2, y_3$ in $\mathcal{A}$ such that $\mathcal{A}^{\otimes 3}$ be a $3$-Leibniz module.
\begin{itemize}
\item
If $\mathcal{A}$ be the first $3$-Leibniz algebra.
{\small
\begin{align}
&\rho_{1}:\mathcal{A}^{\otimes 3}\otimes \mathcal{A}^{\otimes 2}\longrightarrow \mathcal{A}^{\otimes 3},\nonumber\\
&\rho_{1}(y_{1}\otimes y_{2}\otimes y_{3},x_{2},x_{3}):=\mathrm{ad}_{(\widehat{x_{1}},x_{2},x_{3})}^{(3)}(y_{1}\otimes y_{2}\otimes y_{3}),\nonumber\\
&=(\mathrm{ad}_{(\widehat{x_{1}},x_{2},x_{3})}\otimes 1\otimes 1+1\otimes \mathrm{ad}_{(\widehat{x_{1}},x_{2},x_{3})}\otimes 1+1\otimes 1\otimes \mathrm{ad}_{(\widehat{x_{1}},x_{2},x_{3})})(y_{1}\otimes y_{2}\otimes y_{3})\nonumber\\
&=[y_{1},x_{2},x_{3}]\otimes y_{2}\otimes y_{3}+y_{1}\otimes [y_{2},x_{2},x_{3}]\otimes y_{3}+y_{1}\otimes y_{2}\otimes [y_{3},x_{2},x_{3}],\nonumber\\
&\rho_{2}:\mathcal{A}\otimes \mathcal{A}^{\otimes 3}\otimes \mathcal{A}\longrightarrow \mathcal{A}^{\otimes 3},\nonumber\\
&\rho_{2}(x_{1},y_{1}\otimes y_{2}\otimes y_{3},x_{3})=0,\nonumber\\
&\rho_{3}:\mathcal{A}^{\otimes 2}\otimes \mathcal{A}^{\otimes 3}\longrightarrow \mathcal{A}^{\otimes 3},\nonumber\\
&\rho_{3}(x_{1},x_{2},y_{1}\otimes y_{2}\otimes y_{3})=0.\label{r1f}
\end{align}
}
{\small
\begin{align}
&\rho_{1}:\mathcal{A}^{\otimes 3}\otimes \mathcal{A}^{\otimes 2}\longrightarrow \mathcal{A}^{\otimes 3},\nonumber\\
&\rho_{1}(y_{1}\otimes y_{2}\otimes y_{3},x_{2},x_{3}):=(\mathrm{ad}_{(\widehat{x_{1}},x_{2},x_{3})}\otimes 1\otimes 1)(y_{1}\otimes y_{2}\otimes y_{3})
=[y_{1},x_{2},x_{3}]\otimes y_{2}\otimes y_{3},\nonumber\\
&\rho_{2}:\mathcal{A}\otimes \mathcal{A}^{\otimes 3}\otimes \mathcal{A}\longrightarrow \mathcal{A}^{\otimes 3},\nonumber\\
&\rho_{2}(x_{1},y_{1}\otimes y_{2}\otimes y_{3},x_{3})=(\mathrm{ad}_{(x_{1},\widehat{x_{2}},x_{3})}\otimes 1\otimes 1)(y_{1}\otimes y_{2}\otimes y_{3})
=[x_{1},y_{1},x_{3}]\otimes y_{2}\otimes y_{3},\nonumber\\
&\rho_{3}:\mathcal{A}^{\otimes 2}\otimes \mathcal{A}^{\otimes 3}\longrightarrow \mathcal{A}^{\otimes 3},\nonumber\\
&\rho_{3}(x_{1},x_{2},y_{1}\otimes y_{2}\otimes y_{3})=(\mathrm{ad}_{(x_{1},x_{2},\widehat{x_{3}})}\otimes 1\otimes 1)(y_{1}\otimes y_{2}\otimes y_{3})
=[x_{1},x_{2},y_{1}]\otimes y_{2}\otimes y_{3}.\label{r2f}
\end{align}
}
{\small
\begin{align}
&\rho_{1}:\mathcal{A}^{\otimes 3}\otimes \mathcal{A}^{\otimes 2}\longrightarrow \mathcal{A}^{\otimes 3},\nonumber\\
&\rho_{1}(y_{1}\otimes y_{2}\otimes y_{3},x_{2},x_{3}):=(1\otimes \mathrm{ad}_{(\widehat{x_{1}},x_{2},x_{3})}\otimes 1)(y_{1}\otimes y_{2}\otimes y_{3})
=y_{1}\otimes [y_{2},x_{2},x_{3}]\otimes y_{3},\nonumber\\
&\rho_{2}:\mathcal{A}\otimes \mathcal{A}^{\otimes 3}\otimes \mathcal{A}\longrightarrow \mathcal{A}^{\otimes 3},\nonumber\\
&\rho_{2}(x_{1},y_{1}\otimes y_{2}\otimes y_{3},x_{3})=(1\otimes \mathrm{ad}_{(x_{1},\widehat{x_{2}},x_{3})}\otimes 1)(y_{1}\otimes y_{2}\otimes y_{3})
=y_{1}\otimes [x_{1},y_{2},x_{3}]\otimes y_{3},\nonumber\\
&\rho_{3}:\mathcal{A}^{\otimes 2}\otimes \mathcal{A}^{\otimes 3}\longrightarrow \mathcal{A}^{\otimes 3},\nonumber\\
&\rho_{3}(x_{1},x_{2},y_{1}\otimes y_{2}\otimes y_{3})=(1\otimes \mathrm{ad}_{(x_{1},x_{2},\widehat{x_{3}})}\otimes 1)(y_{1}\otimes y_{2}\otimes y_{3})
=y_{1}\otimes [x_{1},x_{2},y_{2}]\otimes y_{3}.\label{r3f}
\end{align}
}
{\small
\begin{align}
&\rho_{1}:\mathcal{A}^{\otimes 3}\otimes \mathcal{A}^{\otimes 2}\longrightarrow \mathcal{A}^{\otimes 3},\nonumber\\
&\rho_{1}(y_{1}\otimes y_{2}\otimes y_{3},x_{2},x_{3}):=(1\otimes 1 \otimes \mathrm{ad}_{(\widehat{x_{1}},x_{2},x_{3})})(y_{1}\otimes y_{2}\otimes y_{3})=y_{1}\otimes y_{2}\otimes [y_{3},x_{2},x_{3}],\nonumber\\
&\rho_{2}:\mathcal{A}\otimes \mathcal{A}^{\otimes 3}\otimes \mathcal{A}\longrightarrow \mathcal{A}^{\otimes 3},\nonumber\\
&\rho_{2}(x_{1},y_{1}\otimes y_{2}\otimes y_{3},x_{3})=(1\otimes 1 \otimes \mathrm{ad}_{(x_{1},\widehat{x_{2}},x_{3})})(y_{1}\otimes y_{2}\otimes y_{3})=y_{1}\otimes y_{2}\otimes [x_{1},y_{3},x_{3}],\nonumber\\
&\rho_{3}:\mathcal{A}^{\otimes 2}\otimes \mathcal{A}^{\otimes 3}\longrightarrow \mathcal{A}^{\otimes 3},\nonumber\\
&\rho_{3}(x_{1},x_{2},y_{1}\otimes y_{2}\otimes y_{3})=(1\otimes 1 \otimes \mathrm{ad}_{(x_{1},x_{2},\widehat{x_{3}})})(y_{1}\otimes y_{2}\otimes y_{3})=y_{1}\otimes y_{2}\otimes [x_{1},x_{2},y_{3}].\label{r4f}
\end{align}
}
It is easy to check that  $\mathcal{A}^{\otimes 3}$ with above actions is a $3$-Leibniz module.
\item
 If $\mathcal{A}$ be the second $3$-Leibniz algebra then we will have the actions \eqref{r2f}-\eqref{r4f} and the following action:
{\small
\begin{align}
&\rho_{1}:\mathcal{A}^{\otimes 3}\otimes \mathcal{A}^{\otimes 2}\longrightarrow \mathcal{A}^{\otimes 3},\nonumber\\
&\rho_{1}(y_{1}\otimes y_{2}\otimes y_{3},x_{2},x_{3})=0,\nonumber\\
&\rho_{2}:\mathcal{A}\otimes \mathcal{A}^{\otimes 3}\otimes \mathcal{A}\longrightarrow \mathcal{A}^{\otimes 3},\nonumber\\
&\rho_{2}(x_{1},y_{1}\otimes y_{2}\otimes y_{3},x_{3}):=\mathrm{ad}_{(x_{1},\widehat{x_{2}},x_{3})}^{(3)}(y_{1}\otimes y_{2}\otimes y_{3})\nonumber\\
&=(\mathrm{ad}_{(x_{1},\widehat{x_{2}},x_{3})}\otimes 1\otimes 1+1\otimes \mathrm{ad}_{(x_{1},\widehat{x_{2}},x_{3})}\otimes 1+1\otimes 1\otimes \mathrm{ad}_{(x_{1},\widehat{x_{2}},x_{3})})(y_{1}\otimes y_{2}\otimes y_{3})\nonumber\\
&=[x_{1},y_{1},x_{3}]\otimes y_{2}\otimes y_{3}+y_{1}\otimes [x_{1},y_{2},x_{3}]\otimes y_{3}+y_{1}\otimes y_{2}\otimes [x_{1},y_{3},x_{3}],\nonumber\\
&\rho_{3}:\mathcal{A}^{\otimes 2}\otimes \mathcal{A}^{\otimes 3}\longrightarrow \mathcal{A}^{\otimes 3},\nonumber\\
&\rho_{3}(x_{1},x_{2},y_{1}\otimes y_{2}\otimes y_{3})=0.\label{r1s}
\end{align}
}
It is easily seen that $\mathcal{A}^{\otimes 3}$ is a $3$-Leibniz module.
\item
 If $\mathcal{A}$ be the third $3$-Leibniz algebra then we will have the actions \eqref{r2f}-\eqref{r4f} and the following action:
{\small
\begin{align}
&\rho_{1}:\mathcal{A}^{\otimes 3}\otimes \mathcal{A}^{\otimes 2}\longrightarrow \mathcal{A}^{\otimes 3},\nonumber\\
&\rho_{1}(y_{1}\otimes y_{2}\otimes y_{3},x_{2},x_{3})=0,\nonumber\\
&\rho_{2}:\mathcal{A}\otimes \mathcal{A}^{\otimes 3}\otimes \mathcal{A}\longrightarrow \mathcal{A}^{\otimes 3},\nonumber\\
&\rho_{2}(x_{1},y_{1}\otimes y_{2}\otimes y_{3},x_{3})=0,\nonumber\\
&\rho_{3}(x_{1},y_{1}\otimes y_{2}\otimes y_{3},x_{3}):=\mathrm{ad}_{(x_{1},x_{2},\widehat{x_{3}})}^{(3)}(y_{1}\otimes y_{2}\otimes y_{3})\nonumber\\
&=(\mathrm{ad}_{(x_{1},x_{2},\widehat{x_{3}})}\otimes 1\otimes 1+1\otimes \mathrm{ad}_{(x_{1},x_{2},\widehat{x_{3}})}\otimes 1+1\otimes 1\otimes \mathrm{ad}_{(x_{1},x_{2},\widehat{x_{3}})})(y_{1}\otimes y_{2}\otimes y_{3})\nonumber\\
&=[x_{1},x_{2},y_{1}]\otimes y_{2}\otimes y_{3}+y_{1}\otimes [x_{1},x_{2},y_{2}]\otimes y_{3}+y_{1}\otimes y_{2}\otimes [x_{1},x_{2},y_{3}].\label{r2t}
\end{align}
}
It is obvious that $\mathcal{A}$ is a $3$-Leibniz module.
\end{itemize}
In the following definition we suppose that $\mathcal{A}$ be a $3$-Leibniz algebra and
$\mathfrak{g}=\mathcal{A}\otimes \mathcal{A}$ be its associated Leibniz algebra and $M$ be a representation of $\mathcal{A}$. We generalize the $p$-cochain of $\mathcal{A}$
$(p\geq 1)$ with coefficients in $\mathcal{A}$ to  $p$-cochain of $\mathcal{A}$ $(p\geq 1)$ with coefficients in $M$ and also the corresponding coboundary map is defined.
\begin{dfn} Since we have three types of $3$-Leibniz algebra we define the cohomology complexs for them separately.
\begin{enumerate}
\item
If $\mathcal{A}$ be the first $3$-Leibniz algebra then $\mathfrak{g}=\mathcal{A}\otimes \mathcal{A}$ with the following bracket
\begin{align*}
[x_1\otimes x_2,y_1\otimes y_2]=[x_1,y_1,y_2]\otimes x_2+x_1\otimes[x_2,y_1,y_2],
\end{align*}
will be a right Leibniz algebra. The  $p$-cochain of $\mathcal{A}$
$(p\geq 1)$ with the coefficients in $M$ is a linear map from $\mathcal{A}\otimes
\mathfrak{g}^{\otimes (p-1)}$ to $M$. Setting  $\Gamma L^{0}(\mathcal{A},M):=\mathcal{A}\otimes M$ the space of $p$-cochains is denoted by
$\Gamma L^{p}(\mathcal{A},M)$. The coboundary map is given  by
\begin{align*}
&d^{p}:\Gamma L^{p}(\mathcal{A},M)\longrightarrow \Gamma L^{p+1}(\mathcal{A},M)\\
&d^{0}(x\otimes m)(y)=-\rho_3(y,x,m),\quad\forall x,y\in\mathcal{A},\forall m\in M
\end{align*}
{\small
\begin{align}
d^{p}(\alpha)(Y,X_{1}, . . ., X_{p-1})&=\sum_{i=1}^{p-2}
\sum_{j=i+1}^{p-1}(-1)^{i}\alpha(Y,X_{1}, . . ., \widehat{X_{i}}
, . . ., X_{j-1},[X_{i},X_{j}],X_{j+1}, . . ., X_{p-1})\nonumber\\
&\;\;+\sum_{i=1}^{p-1}(-1)^{i}\alpha(\{X_{i},Y\},X_{1}, . . ., \widehat{X_{i}}, . . ., X_{p-1})\nonumber\\
&\;\;+(-1)^{p}\alpha([y_1,y_2,y_3],X_{1}, . . ., X_{p-1})\nonumber\\
&\;\;+\sum_{i=1}^{p-1}(-1)^{i+1}\rho_1(\alpha(Y,X_{1}, . . ., \widehat{X_{i}}, . . ., X_{p-1}),X_i)\nonumber\\
&\;\;+(-1)^{p+1}\sum_{i=1}^{3}\rho_i(y_{1}, . . ., y_{i-1},\alpha(y_i,X_{1}, . . .,  X_{p-1}), . . ., y_{3}),
\end{align}
}
where $X_{i}\in \mathfrak{g}$ $(i=1, . . ., p-1)$, $Y=y_{1}\otimes y_2\otimes y_3\in\mathcal{A}^{\otimes 3}$ and for $X_i=x_i^1\otimes x_i^2$ we set $\{X_i,Y\}:=\sum_{i=1}^{3}y_{1}\otimes . . . \otimes y_{i-1}\otimes[y_i, x_i^1, x_i^2]\otimes . . . \otimes y_3$. In this case we have
{\small
\begin{align*}
& d^{1}:\Gamma L^{1}(\mathcal{A},M)\longrightarrow \Gamma L^{2}(\mathcal{A},M)\\
&d^{1}(\alpha)(y_{1}\otimes y_{2}\otimes y_{3})=-\alpha([y_{1},y_{2},y_{3}])+\rho_{1}(\alpha(y_{1}),y_{2},y_{3})+\rho_{2}(y_{1},\alpha(y_{2}),y_{3})+\rho_{3}(y_{1},y_{2},\alpha(y_{3}))
\end{align*}
}
\item
If $\mathcal{A}$ be the second $3$-Leibniz algebra then $\mathfrak{g}=\mathcal{A}\otimes \mathcal{A}$ with the following bracket
\begin{align*}
[x_1\otimes x_2,y_1\otimes y_2]=[x_1,y_1, x_2]\otimes y_2+y_1\otimes[x_1,y_2, x_2],
\end{align*}
will be a left Leibniz algebra and with the following bracket will be a right Leibniz algebra.
\begin{align*}
[x_1\otimes x_2,y_1\otimes y_2]=[y_1, x_1,y_2]\otimes x_2+x_1\otimes[y_1, x_2,y_2].
\end{align*}
The  $p$-cochain of $\mathcal{A}$
$(p\geq 2)$ with the coefficients in $M$ is a linear map from $\mathcal{A}^{\otimes 3}\otimes\mathfrak{g}^{\otimes (p-2)}$ to $M$. Set also $\Gamma L^{0}(\mathcal{A},M):=M\otimes\mathcal{A}$ and $\Gamma L^{1}(\mathcal{A},M)$ is a linear map from
$\mathcal{A}$ to $M$.  We will denote by $\Gamma L^{p}(\mathcal{A},M)$ the space of $p$-cochains. The coboundary map is given  by
\begin{align*}
&d^{p}:\Gamma L^{p}(\mathcal{A},M)\longrightarrow \Gamma L^{p+1}(\mathcal{A},M)\\
&d^{0}(m\otimes x)(y)=-\rho_1(m,y,x),\quad\forall x,y\in\mathcal{A},\forall m\in M
\end{align*}
{\small
\begin{align}
d^{p}(\alpha)(Y,X_{1}, . . ., X_{p-1})
&=\sum_{i=1}^{p-2}\sum_{j=i+1}^{p-1}(-1)^{i}\alpha(Y,X_{1}, . . ., \widehat{X_{i}}, . . ., X_{j-1},[X_{i},X_{j}],X_{j+1}, . . ., X_{p-1})\nonumber\\
&\;\;+\sum_{i=1}^{p-1}(-1)^{i}\alpha(\{X_{i},Y\},X_{1}, . . ., \widehat{X_{i}}, . . ., X_{p-1})\nonumber\\
&\;\;+(-1)^{p}\alpha(x_1^{1},[y_1,y_2,y_3],x_1^{2},X_2, . . ., X_{p-1})\nonumber\\
&\;\;+\sum_{i=1}^{p-1}(-1)^{i+1}\rho_2(x_i^{1},\alpha(Y,X_{1}, . . ., \widehat{X_{i}}, . . ., X_{p-1}),x_i^{2})\nonumber\\
&\;\;+(-1)^{p+1}\sum_{i=1}^{3}\rho_i(y_{1}, . . ., y_{i-1},\alpha(x_i^{1},y_i, x_i^{2}, . . .,  X_{p-1}), . . ., y_{3}),
\end{align}
}
where $X_{i}\in \mathfrak{g}$ $(i=1, . . ., p-1)$, $Y=y_{1}\otimes y_2\otimes y_3\in\mathcal{A}^{\otimes 3}$ and for $X_i=x_i^{1}\otimes x_i^{2}$ we set $\{X_i,Y\}:=\sum_{j=1}^{3}y_{1}\otimes . . . \otimes y_{j-1}\otimes[x_i^{1},y_j, x_i^{2}]\otimes . . . \otimes y_3$. In this case we have
{\small
\begin{align*}
& d^{1}:\Gamma L^{1}(\mathcal{A},M)\longrightarrow \Gamma L^{2}(\mathcal{A},M)\nonumber\\
&d^{1}(\alpha)(y_{1}\otimes y_{2}\otimes y_{3})=-\alpha([y_{1},y_{2},y_{3}])+\rho_{1}(\alpha(y_{1}),y_{2},y_{3})+\rho_{2}(y_{1},\alpha(y_{2}),y_{3})+\rho_{3}(y_{1},y_{2},\alpha(y_{3}))
\end{align*}
}
\item
If $\mathcal{A}$ be the third $3$-Leibniz algebra then $\mathfrak{g}=\mathcal{A}\otimes \mathcal{A}$ with the following bracket
\begin{align*}
[x_1\otimes x_2,y_1\otimes y_2]=[x_1, x_2,y_1]\otimes y_2+y_1\otimes[x_1, x_2,y_2],
\end{align*}
will be a left Leibniz algebra. The  $p$-cochain of $\mathcal{A}$
$(p\geq 1)$ with coefficients in $M$ is a linear map from $\mathfrak{g}^{\otimes (p-1)}\otimes \mathcal{A}$ to $M$. Set also $\Gamma L^{0}(\mathcal{A},M):=M\otimes\mathcal{A}$. The space of $p$-cochains is denoted by
$\Gamma L^{p}(\mathcal{A},M)$. The coboundary map is given  by
\begin{align*}
&	d^{p}:\Gamma L^{p}(\mathcal{A},M)\longrightarrow \Gamma L^{p+1}(\mathcal{A},M)\\
&d^{0}(m\otimes x)(y)=-\rho_1(m,x,y),\quad\forall x,y\in\mathcal{A},\forall m\in M
\end{align*}
{\small
\begin{align}
d^{p}(\alpha)(X_{1}, . . ., X_{p-1},Y)
&=\sum_{i=1}^{p-2}\sum_{j=i+1}^{p-1}(-1)^{i}\alpha(X_{1}, . . ., \widehat{X_{i}}, . . ., X_{j-1},[X_{i},X_{j}],X_{j+1}, . . ., X_{p-1},Y)\nonumber\\
&\;\;+\sum_{i=1}^{p-1}(-1)^{i}\alpha(X_{1}, . . ., \widehat{X_{i}}, . . ., X_{p-1},\{X_{i},Y\})\nonumber\\
&\;\;+(-1)^{p}\alpha(X_{1}, . . ., X_{p-1},[y_1,y_2,y_3])\nonumber\\
&\;\;+\sum_{i=1}^{p-1}(-1)^{i+1}\rho_3(X_i,\alpha(X_{1}, . . ., \widehat{X_{i}}, . . ., X_{p-1},Y))\nonumber\\
&\;\;+(-1)^{p+1}\sum_{i=1}^{3}\rho_i(y_{1}, . . ., y_{i-1},\alpha(X_{1}, . . .,  X_{p-1},y_i), . . ., y_{3}),
\end{align}
}
where $X_{i}\in \mathfrak{g}$  $(i=1, . . ., p-1)$, $Y=y_{1}\otimes y_2\otimes y_3\in\mathcal{A}^{\otimes 3}$ and for
$X_i=x_i^{1}\otimes x_i^{2}$ we set $\{X_i,Y\}:=\sum_{j=1}^{3}y_{1}\otimes . . . \otimes y_{j-1}\otimes[x_i^{1},x_i^{2},y_j]\otimes . . . \otimes y_3$. In this case we have
{\small
\begin{align*}
& d^{1}:\Gamma L^{1}(\mathcal{A},M)\longrightarrow \Gamma L^{2}(\mathcal{A},M)\\
&d^{1}(\alpha)(y_{1}\otimes y_{2}\otimes y_{3})=-\alpha([y_{1},y_{2},y_{3}])+\rho_{1}(\alpha(y_{1}),y_{2},y_{3})+\rho_{2}(y_{1},\alpha(y_{2}),y_{3})+\rho_{3}(y_{1},y_{2},\alpha(y_{3}))
\end{align*}
}
\end{enumerate}	
\end{dfn}
Now, with these actions we define the $3$-Leibniz bialgebra.
\begin{dfn}A {\textit{$3$-Leibniz bialgebra}} $(\mathcal{A},\gamma)$ is (the first or the second or the third) $3$-Leibniz algebra $\mathcal{A}$ with a linear map (cocommutator) $\gamma:\mathcal{A}\longrightarrow \mathcal{A}^{\otimes 3} $ such that
\begin{itemize}
\item
$\gamma$ is a $1$-cocycle on $\mathcal{A}$ with values in $\mathcal{A}^{\otimes 3}$
according to \eqref{r1f}-\eqref{r4f}, \eqref{r1s} and \eqref{r2t}.
{\small
\begin{align}
\gamma[y_{1},y_{2},y_{3}]=\rho_{1}(\gamma(y_{1}),y_{2},y_{3})+\rho_{2}(y_{1},\gamma(y_{2}),y_{3})+\rho_{3}(y_{1},y_{2},\gamma(y_{3})),\label{1-cocycle}
\end{align}
}
where $\mathcal{A}^{\otimes 3}$ be a $3$-Leibniz module. In above identity $\rho_{1}$,$\rho_{2}$ and $\rho_{3}$ are  three actions which transform $\mathcal{A}^{\otimes 3}$ to a $3$-Leibniz module.
\item
{\small 
$\gamma ^{t}:{\mathcal{A}^{\ast}}^{\otimes 3}\longrightarrow \mathcal{A}^{\ast}$
}
defines a $3$-Leibniz bracket on $\mathcal{A}^{\ast}$.
\end{itemize}
With the notation
{\small
\begin{align}	
[{\widetilde{x}}^{1},{\widetilde{x}}^{2},{\widetilde{x}}^{3}]_{\ast}=\gamma^{t}({\widetilde{x}}^{1}\otimes {\widetilde{x}}^{2}\otimes {\widetilde{x}}^{3}),
\quad\forall{\widetilde{x}}^{1},{\widetilde{x}}^{2},{\widetilde{x}}^{3}\in\mathcal{A}^{\ast}\label{notation}
\end{align}
}
 $\forall x\in\mathcal{A}$ we have
{\small
\begin{align}
\langle [{\widetilde{x}}^{1},{\widetilde{x}}^{2},{\widetilde{x}}^{3}]_{\ast},x\rangle=\langle\gamma^{t}({\widetilde{x}}^{1}\otimes {\widetilde{x}}^{2}\otimes {\widetilde{x}}^{3}),x\rangle=\langle{\widetilde{x}}^{1}\otimes {\widetilde{x}}^{2}\otimes {\widetilde{x}}^{3},\gamma(x) \rangle, \label{pairing}
\end{align}
}
where $\langle ,\rangle$ is the natural pairing between $\mathcal{A}$ and $\mathcal{A}^{\ast}$.
\end{dfn}
As for the type of the $3$-Leibniz algebra $\mathcal{A}$ and also its actions $\rho_{1}$ , $\rho_{2}$ and $\rho_{3}$, the $1$-cocycle condition \eqref{1-cocycle} can be rewritten in  one of the following forms:
{\small
\begin{align}
\gamma([x_{1},x_{2},x_{3}])&= \mathrm{ad}_{(\widehat{x_{1}},x_{2},x_{3})}^{(3)}
\gamma(x_{1}),\label{1-cocycle1}\\
\nonumber\\
\gamma([x_{1},x_{2},x_{3}])&= \mathrm{ad}_{(x_{1},\widehat{x_{2}},x_{3})}^{(3)}
\gamma(x_{2}),\label{1-cocycle2}\\
\nonumber\\
\gamma([x_{1},x_{2},x_{3}])&= \mathrm{ad}_{(x_{1},x_{2},\widehat{x_{3}})}^{(3)}
\gamma(x_{3}),\label{1-cocycle3}\\
\nonumber\\
\gamma([x_{1},x_{2},x_{3}])&=(\mathrm{ad}_{(\widehat{x_{1}},x_{2},x_{3})}\otimes 1\otimes 1)(\gamma(x_{1}))\nonumber\\
&\;\;+(\mathrm{ad}_{(x_{1},\widehat{x_{2}},x_{3})}\otimes 1\otimes 1)(\gamma(x_{2}))+(\mathrm{ad}_{(x_{1},x_{2},\widehat{x_{3}})}\otimes 1\otimes 1)(\gamma(x_{3})),\label{1-cocycle4}\\
\nonumber\\
\gamma([x_{1},x_{2},x_{3}])&=(1\otimes \mathrm{ad}_{(\widehat{x_{1}},x_{2},x_{3})}\otimes 1)(\gamma(x_{1}))\nonumber\\
&\;\;+(1\otimes \mathrm{ad}_{(x_{1},\widehat{x_{2}},x_{3})}\otimes  1)(\gamma(x_{2}))+(1\otimes \mathrm{ad}_{(x_{1},x_{2},\widehat{x_{3}})}\otimes 1)(\gamma(x_{3})),\label{1-cocycle5}\\
\nonumber\\
\gamma([x_{1},x_{2},x_{3}])&=(1\otimes 1\otimes \mathrm{ad}_{(\widehat{x_{1}},x_{2},x_{3})})(\gamma(x_{1}))\nonumber\\
&\;\;+(1\otimes 1\otimes \mathrm{ad}_{(x_{1},\widehat{x_{2}},x_{3})})(\gamma(x_{2}))+(1\otimes 1\otimes \mathrm{ad}_{(x_{1},x_{2},\widehat{x_{3}})})(\gamma(x_{3})).\label{1-cocycle6}
\end{align}
}
In \eqref{1-cocycle1}, \eqref{1-cocycle2},\eqref{1-cocycle3} $\mathcal{A}$ is the first and the second and the third $3$-Leibniz algebra respectively.
In \eqref{1-cocycle4}, \eqref{1-cocycle5}, \eqref{1-cocycle6} $\mathcal{A}$ can be a $3$-Leibniz algebra of various three types.
According to above $1$-cocycle  conditions the $3$-Leibniz algebra $\mathcal{A}^{\ast}$ can be the first or the second or the third $3$-Leibniz algebra. We investigate this subject as follows:
{\textit{
\begin{prop}
If $(\mathcal{A},\gamma)$ be a  $3$-Leibniz bialgebra, and $\mu$ be a  $3$-Leibniz bracket on $\mathcal{A}$, then
$(\mathcal{A}^{\ast},\mu^{t})$ will be a $3$-Leibniz bialgebra, where $\gamma^{t}$ is a $3$-Leibniz bracket on $\mathcal{A}^{\ast}$.
\end{prop}
}}
\begin{proof}
The proof will be divided into three steps.
\begin{itemize}
\item
Assume \eqref{1-cocycle1} holds for the value of $\gamma([x_{1},x_{2},x_{3}])$, then from \eqref{pairing} we have
\begin{align}
\langle [\xi_{1},\xi_{2},\xi_{3}]_{\ast},[x_{1},x_{2},x_{3}]\rangle&=\langle \xi_{1}\otimes \xi_{2}\otimes\xi_{3},\gamma[x_{1},x_{2},x_{3}]\rangle\nonumber\\
&=\langle\xi_{1}\otimes\xi_{2}\otimes\xi_{3},(\mathrm{ad}_{(\widehat{x_{1}},x_{2},x_{3})}\otimes 1\otimes 1)(\gamma (x_{1}))\rangle\nonumber\\
&\;\;+\langle\xi_{1}\otimes\xi_{2}\otimes\xi_{3},(1\otimes \mathrm{ad}_{(\widehat{x_{1}},x_{2},x_{3})}\otimes 1)(\gamma(x_{1}))\rangle\nonumber\\
&\;\;+\langle\xi_{1}\otimes\xi_{2}\otimes\xi_{3},(1\otimes 1\otimes \mathrm{ad}_{(\widehat{x_{1}},x_{2},x_{3})})(\gamma(x_{1}))\rangle.\label{1'}
\end{align}
We now define the coadjoint representation of a $3$-Leibniz algebra on the dual vector space. Let $\mathcal{A}$ be a $3$-Leibniz algebra and let $\mathcal{A}^{\ast}$ be its dual vector space, then for $x_{1},x_{2},x_{3}\in\mathcal{A}$ we have
{\small
\begin{align}
&\mathrm{ad}_{(\widehat{x_{1}},x_{2},x_{3})}^{\ast}:\mathcal{A}^{\ast}\longrightarrow \mathcal{A}^{\ast},\nonumber\\
&\langle \mathrm{ad}_{(\widehat{x_{1}},x_{2},x_{3})}^{\ast}\xi,y\rangle=\langle\xi,\mathrm{ad}_{(\widehat{x_{1}},x_{2},x_{3})}y\rangle=\langle\xi,[y,x_{2},x_{3}]\rangle,
\end{align}
and similarly
\begin{align}
&	\mathrm{ad}_{(x_{1},\widehat{x_{2}},x_{3})}^{\ast}:\mathcal{A}^{\ast}\longrightarrow \mathcal{A}^{\ast},\nonumber\\
&\langle \mathrm{ad}_{(x_{1},\widehat{x_{2}},x_{3})}^{\ast}\xi,y\rangle
=\langle\xi,[x_{1},y,x_{3}]\rangle,
\end{align}
\begin{align}
&	\mathrm{ad}_{(x_{1},x_{2},\widehat{x_{3}})}^{\ast}:\mathcal{A}^{\ast}\longrightarrow \mathcal{A}^{\ast},\nonumber\\
&\langle \mathrm{ad}_{(x_{1},x_{2},\widehat{x_{3}})}^{\ast}\xi,y\rangle
=\langle\xi,[x_{1},x_{2},y]\rangle.
\end{align}
}
Using these relations, \eqref{1'} can be rewritten as
\begin{align}
\begin{split}
&\langle [\xi_{1},\xi_{2},\xi_{3}]_{\ast},[x_{1},x_{2},x_{3}]\rangle=\langle[\mathrm{ad}_{(\widehat{x_{1}},x_{2},x_{3})}^{\ast}\xi_{1},\xi_{2},\xi_{3}]_{\ast},x_{1}\rangle\\
&+\langle[\xi_{1},\mathrm{ad}_{(\widehat{x_{1}},x_{2},x_{3})}^{\ast}\xi_{2},\xi_{3}]_{\ast},x_{1}\rangle+\langle[\xi_{1},\xi_{2},\mathrm{ad}_{(\widehat{x_{1}},x_{2},x_{3})}^{\ast}\xi_{3}]_{\ast},x_{1}\rangle.\label{1''}
\end{split}
\end{align}
In the similar way as above; for any $\xi_{1},\xi_{2},\xi_{3}\in\mathcal{A}^{\ast}$ we have
{\small
\begin{align}
&	\mathrm{ad}_{(\widehat{\xi_{1}},\xi_{2},\xi_{3})}^{\ast}:\mathcal{A}\longrightarrow \mathcal{A}\cong{\mathcal{A}^{\ast}}^{\ast},\nonumber\\
&	\langle \mathrm{ad}_{(\widehat{\xi_{1}},\xi_{2},\xi_{3})}^{\ast}x,\eta\rangle=\langle x,\mathrm{ad}_{(\widehat{\xi_{1}},\xi_{2},\xi_{3})}\eta\rangle=\langle x,[\eta,\xi_{2},\xi_{3}]_{\ast}\rangle,
\end{align}
and similarly
\begin{align}
&	\mathrm{ad}_{(\xi_{1},\widehat{\xi_{2}},\xi_{3})}^{\ast}:\mathcal{A}\longrightarrow \mathcal{A}\cong{\mathcal{A}^{\ast}}^{\ast},\nonumber\\
&\langle \mathrm{ad}_{(\xi_{1},\widehat{\xi_{2}},\xi_{3})}^{\ast} x,\eta\rangle=\langle x,[\xi_{1},\eta,\xi_{3}]_{\ast}\rangle,
\end{align}
\begin{align}
&	\mathrm{ad}_{(\xi_{1},\xi_{2},\widehat{\xi_{3}})}^{\ast}:\mathcal{A}\longrightarrow\mathcal{A}\cong{\mathcal{A}^{\ast}}^{\ast},\nonumber\\
&	\langle \mathrm{ad}_{(\xi_{1},\xi_{2},\widehat{\xi_{3}})}^{\ast} x,\eta\rangle=\langle x,[\xi_{1},\xi_{2},\eta]_{\ast}\rangle.
\end{align}
}
By using these relations, \eqref{1''} can be rewritten as
\begin{align}
\begin{split}
\langle [\xi_{1},\xi_{2},\xi_{3}]_{\ast},[x_{1},x_{2},x_{3}]\rangle&=\langle \mathrm{ad}_{(\widehat{x_{1}},x_{2},x_{3})}^{\ast}\xi_{1},\mathrm{ad}_{(\widehat{\xi_{1}},\xi_{2},\xi_{3})}^{\ast}x_{1}\rangle\\
&\;\;+\langle \mathrm{ad}_{(\widehat{x_{1}},x_{2},x_{3})}^{\ast}\xi_{2},\mathrm{ad}_{(\xi_{1},\widehat{\xi_{2}},\xi_{3})}^{\ast}x_{1}\rangle\\
&\;\;+\langle \mathrm{ad}_{(\widehat{x_{1}},x_{2},x_{3})}^{\ast}\xi_{3},\mathrm{ad}_{(\xi_{1},\xi_{2},\widehat{\xi_{3}})}^{\ast}x_{1}\rangle,\label{1'''}
\end{split}
\end{align}
or
\begin{align}
\begin{split}
\langle [\xi_{1},\xi_{2},\xi_{3}]_{\ast},\mu(x_{1}\otimes x_{2}\otimes x_{3})\rangle&=\langle
\xi_{1},[\mathrm{ad}_{(\widehat{\xi_{1}},\xi_{2},\xi_{3})}^{\ast}x_{1},x_{2},x_{3}]\rangle\\
&\;\;+\langle\xi_{2},[\mathrm{ad}_{(\xi_{1},\widehat{\xi_{2}},\xi_{3})}^{\ast}x_{1},x_{2},x_{3}]\rangle\\
&\;\;+\langle\xi_{3},[\mathrm{ad}_{(\xi_{1},\xi_{2},\widehat{\xi_{3}})}^{\ast}x_{1},x_{2},x_{3}]\rangle,
\end{split}
\end{align}
where $\mu$ is the $3$-Leibniz bracket on $\mathcal{A}$ and $\mu^{t}:\mathcal{A}^{\ast}\longrightarrow \mathcal{A}^{\ast}\otimes \mathcal{A}^{\ast}\otimes \mathcal{A}^{\ast}$ is a cocommutator on $\mathcal{A}^{\ast}$. Therefore, we have
{\small
\begin{align}
\begin{split}
\langle \mu^{t}[\xi_{1},\xi_{2},\xi_{3}]_{\ast},x_{1}\otimes x_{2}\otimes x_{3}\rangle&=\langle (\mathrm{ad}_{(\widehat{\xi_{1}},\xi_{2},\xi_{3})}\otimes 1\otimes 1)(\mu^{t}(\xi_{1})),x_{1}\otimes x_{2}\otimes x_{3}\rangle\\
&\;\;+\langle (\mathrm{ad}_{(\xi_{1},\widehat{\xi_{2}},\xi_{3})}\otimes 1\otimes 1)(\mu^{t}(\xi_{2})),x_{1}\otimes x_{2}\otimes x_{3}\rangle\\
&\;\;+\langle (\mathrm{ad}_{(\xi_{1},\xi_{2},\widehat{\xi_{3}})}\otimes 1\otimes 1)(\mu^{t}(\xi_{3})),x_{1}\otimes x_{2}\otimes x_{3}\rangle,
\end{split}
\end{align}
}
or
{\small
\begin{align}
\begin{split}
\mu^{t}[\xi_{1},\xi_{2},\xi_{3}]_{\ast}&=(\mathrm{ad}_{(\widehat{\xi_{1}},\xi_{2},\xi_{3})}\otimes 1\otimes 1)(\mu^{t}(\xi_{1}))+(\mathrm{ad}_{(\xi_{1},\widehat{\xi_{2}},\xi_{3})}\otimes 1\otimes 1)(\mu^{t}(\xi_{2}))\\
&\;\;+(\mathrm{ad}_{(\xi_{1},\xi_{2},\widehat{\xi_{3}})}\otimes 1\otimes 1)(\mu^{t}(\xi_{3})).\label{dual11-cocycle}
\end{split}
\end{align}
}
But, this relation is the $1$-cocycle condition \eqref{1-cocycle4} for $(\mathcal{A}^{\ast},\mu^{t})$ which it  shows ${\mathcal{A}^{\ast}}^{\otimes 3}$ is a
$3$-Leibniz module on $\mathcal{A}^{\ast}$; i.e. {\textit{$\mathcal{A}^{\ast}$ can be a $3$-Leibniz algebra of various three types.}}
\item
In the same way, if \eqref{1-cocycle2} holds for the value of $\gamma([x_{1},x_{2},x_{3}])$, then by assuming that {\textit{$\mathcal{A}$ is the second $3$-Leibniz algebra}},
instead of \eqref{dual11-cocycle} we will have
{\small
\begin{align}
\begin{split}
\mu^{t}[\xi_{1},\xi_{2},\xi_{3}]_{\ast}&=(1\otimes \mathrm{ad}_{(\widehat{\xi_{1}},\xi_{2},\xi_{3})}\otimes 1)(\mu^{t}(\xi_{1}))+(1\otimes \mathrm{ad}_{(\xi_{1},\widehat{\xi_{2}},\xi_{3})}\otimes 1)(\mu^{t}(\xi_{2}))\\
&\;\;+(1\otimes \mathrm{ad}_{(\xi_{1},\xi_{2},\widehat{\xi_{3}})}\otimes  1)(\mu^{t}(\xi_{3})),\label{dual21-cocycle}
\end{split}
\end{align}
}
where this relation is the $1$-cocycle condition \eqref{1-cocycle5} for $(\mathcal{A}^{\ast},\mu^{t})$, i. e. $\mathcal{A}^{\ast}\otimes \mathcal{A}^{\ast}\otimes \mathcal{A}^{\ast}$ is a  $3$-Leibniz module on $\mathcal{A}^{\ast}$ and {\textit{$\mathcal{A}^{\ast}$ can be a $3$-Leibniz algebra of various three types.}}\\
On the other hand, for the third $3$-Leibniz algebra $(\mathcal{A},\mu)$ when one uses \eqref{1-cocycle3} for the value $\gamma([x_{1},x_{2},x_{3}])$  we have
{\small
\begin{align}
\begin{split}
\mu^{t}[\xi_{1},\xi_{2},\xi_{3}]_{\ast}&=(1\otimes 1 \otimes \mathrm{ad}_{(\widehat{\xi_{1}},\xi_{2},\xi_{3})})(\mu^{t}(\xi_{1}))+(1\otimes 1 \otimes \mathrm{ad}_{(\xi_{1},\widehat{\xi_{2}},\xi_{3})})(\mu^{t}(\xi_{2}))\\
&\;\;+(1\otimes 1\otimes \mathrm{ad}_{(\xi_{1},\xi_{2},\widehat{\xi_{3}})} )(\mu^{t}(\xi_{3})),\label{dual31-cocycle}
\end{split}
\end{align}
}
instead of \eqref{dual11-cocycle}, and this shows that $\mu^{t}$ is a $1$-cocycle condition \eqref{1-cocycle6} for $(\mathcal{A}^{\ast},\mu^{t})$ where $\mathcal{A}^{\ast}\otimes \mathcal{A}^{\ast}\otimes \mathcal{A}^{\ast}$ is a $3$-Leibniz module on $\mathcal{A}^{\ast}$ and {\textit{$\mathcal{A}^{\ast}$ can be $3$-Leibniz algebra of various three types.}}
\item
In the same way, if one uses  \eqref{1-cocycle4} for the value of $\gamma([x_{1},x_{2},x_{3}])$, then by assuming that {\textit{$\mathcal{A}$ can be a $3$-Leibniz algebra of various three types}},
we will have
\begin{align}
	\mu^{t}[\xi_{1},\xi_{2},\xi_{3}]_{\ast}=\mathrm{ad}_{(\widehat{\xi_{1}},\xi_{2},\xi_{3})}^{(3)}\mu^{t}(x_{1}),\label{dual41-cocycle}
\end{align}
and this shows that $\mu^{t}$ is a $1$-cocycle condition \eqref{1-cocycle1} for $(\mathcal{A}^{\ast},\mu^{t})$ where it shows $\mathcal{A}^{\ast}\otimes \mathcal{A}^{\ast}\otimes \mathcal{A}^{\ast}$ is a $3$-Leibniz module on $\mathcal{A}^{\ast}$ and {\textit{$\mathcal{A}^{\ast}$ is the first $3$-Leibniz algebra.}}
\item
Using \eqref{1-cocycle5} for the value of $\gamma([x_{1},x_{2},x_{3}])$, then by assuming that {\textit{$\mathcal{A}$ can be a $3$-Leibniz algebra of various three types}},
we have
\begin{align}
\mu^{t}[\xi_{1},\xi_{2},\xi_{3}]_{\ast}=\mathrm{ad}_{(\xi_{1},\widehat{\xi_{2}},\xi_{3})}^{(3)}\mu^{t}(x_{2}),\label{dual51-cocycle}
\end{align}
and this shows that $\mu^{t}$ is a $1$-cocycle condition \eqref{1-cocycle2} for $(\mathcal{A}^{\ast},\mu^{t})$ where $\mathcal{A}^{\ast}\otimes \mathcal{A}^{\ast}\otimes \mathcal{A}^{\ast}$ is a $3$-Leibniz module on $\mathcal{A}^{\ast}$ and {\textit{$\mathcal{A}^{\ast}$ is the second $3$-Leibniz algebra.}}
\item
Finally, if one uses \eqref{1-cocycle6} for the value of $\gamma([x_{1},x_{2},x_{3}])$, then by assuming that {\textit{$\mathcal{A}$ is a $3$-Leibniz algebra of various three types}},
we will have
\begin{align}
\mu^{t}[\xi_{1},\xi_{2},\xi_{3}]_{\ast}=\mathrm{ad}_{(\xi_{1},\xi_{2},\widehat{\xi_{3}})}^{(3)}\mu^{t}(x_{3}),\label{dual61-cocycle}
\end{align}
and this shows that $\mu^{t}$ is a $1$-cocycle condition \eqref{1-cocycle3} for $(\mathcal{A}^{\ast},\mu^{t})$ where $\mathcal{A}^{\ast}\otimes \mathcal{A}^{\ast}\otimes \mathcal{A}^{\ast}$ is a $3$-Leibniz module on $\mathcal{A}^{\ast}$ and {\textit{$\mathcal{A}^{\ast}$ can be the third $3$-Leibniz algebra.}}
\end{itemize}
Therefore, a $3$-Leibniz bialgebra $(\mathcal{A},\gamma)$ can also be denoted by $(\mathcal{A},\mathcal{A}^{\ast})$.
\end{proof}

There are no Manin triple for $3$-Leibniz bialgebras.

\section{ $3$-Leibniz bialgebra in terms of structure constants; some examples}
	\hspace{13.5pt}
In this section, we obtain some examples of $3$-Leibniz bialgebras. For this purpose, we first rewrite the $1$-cocycle conditions  \eqref{1-cocycle1}-\eqref{1-cocycle3} in terms of structure constants of the $3$-Leibniz algebras $\mathcal{A}$ and $\mathcal{A}^{\ast}$. If we choose $(\{x_i\},f_{ijk}\,^{m})$ and $(\{\widetilde{x}^{i}\},\widetilde{f}^{ijk}\,_{m})$ as the basis and structure constants of $3$-Leibniz algebras $\mathcal{A}$ and $\mathcal{A}^{\ast}$ respectively; then we will have the commutation relations as follows
\begin{align}
[x_i, x_j, x_k]=f_{ijk}\,^{m}x_m,\qquad
[\widetilde{x}^{i}, \widetilde{x}^{j}, \widetilde{x}^{k}]_{\ast}=\widetilde{f}^{ijk}\,_{m}\widetilde{x}^{m}.\label{st-con}
\end{align}
Using \eqref{pairing} we have
{\small
\begin{align}
\begin{split}
\langle \widetilde{x}^{j}\otimes \widetilde{x}^{k}\otimes \widetilde{x}^{m},\gamma(x^{i})\rangle&=\langle \gamma^{t}(\widetilde{x}^{j}\otimes \widetilde{x}^{k}\otimes \widetilde{x}^{m}),x_{i}\rangle=\langle [\widetilde{x}^{j},\widetilde{x}^{k},\widetilde{x}^{m}]_{\ast},x_{i}\rangle\\
&=\langle \widetilde{f}^{jkm}\,_{n}\widetilde{x}^{n},x_{i}\rangle=\widetilde{f}^{jkm}\,_{i},
\end{split}
\end{align}
}
namely
{\small
\begin{align}
\gamma(x_{i})=\widetilde{f}^{jkm}\,_{i}  x_{j}\otimes x_{k}\otimes x_{m}\label{cocommutator}
\end{align}
}
Now using structure constants of $\mathcal{A}$ and \eqref{cocommutator} in the 1-cocycle conditions \eqref{1-cocycle1}-\eqref{1-cocycle3} we obtain the following relations respectively:
{\small
\begin{align}
& f_{isn}\,^{p}\widetilde{f}^{jkm}\,_{p}
=\widetilde{f}^{j'km}\,_{i} f_{j'sn}\,^{j}+
\widetilde{f}^{jk'm}\,_{i} f_{k'sn}\,^{k}+\widetilde{f}^{jkm'}\,_{i} f_{m'sn}\,^{m},\label{1-cocycle1'}\\
& f_{isn}\,^{p}\widetilde{f}^{jkm}\,_{p}
=\widetilde{f}^{j'km}\,_{s} f_{ij'n}\,^{j}+
\widetilde{f}^{jk'm}\,_{s} f_{ik'n}\,^{k}+\widetilde{f}^{jkm'}\,_{s} f_{im'n}\,^{m},\label{1-cocycle2'}\\
& f_{isn}\,^{p}\widetilde{f}^{jkm}\,_{p}
=\widetilde{f}^{j'km}\,_{n} f_{isj'}\,^{j}+
\widetilde{f}^{jk'm}\,_{n} f_{isk'}\,^{k}+\widetilde{f}^{jkm'}\,_{n} f_{ism'}\,^{m}.\label{1-cocycle3'}
\end{align}
}
Note that similar to the Lie bialgebras case \cite{Drin} one can use these three relations as a definition of $3$-Leibniz bialgebra.
\begin{dfn}
If the structure constants of the $3$-Leibniz algebras $\mathcal{A}$ and $\mathcal{A}^{\ast}$ satisfy in relations
\eqref{1-cocycle1'}-\eqref{1-cocycle3'} then $3$-Leibniz algebras $\mathcal{A}$ and $\mathcal{A}^{\ast}$ will construct a $3$-Leibniz bialgebra.
\end{dfn}
To use these relations in the calculations, we must first translate the tensor form of these relations to the matrix forms by using the following adjoint representations
{\small
\begin{align}
& f_{isn}\,^{p}=(\chi_{is})_{n}\,^{p}
=(\mathcal{Y}_{i}\,^{p})_{sn}=f'_{sin}\,^{p}=(\chi'_{si})_{n}\,^{p}
=(\mathcal{Y}'_{s}\,^{p})_{in}\\
&\widetilde{f}^{jkm}\,_{p}=(\widetilde{\chi}^{jk})^{m}\,_{p}=(\widetilde{\mathcal{Y}}^{j}\,_{p})^{km}=\widetilde{f'}^{kjm}\,_{p}
=(\widetilde{\chi'}^{kj})^{m}\,_{p}=(\widetilde{\mathcal{Y}'}^{k}\,_{p})^{jm}.
\end{align}
}
Then,  \eqref{1-cocycle1'}-\eqref{1-cocycle3'} have the following matrix forms respectively:
{\small
\begin{align}
&(\chi_{is})(\widetilde{\chi}^{jk})^{t}=(\mathcal{Y}'_{s}\,^{j})^{t}(\widetilde{\mathcal{Y}'}^{k}\,_{i})+
(\mathcal{Y}'_{s}\,^{k})^{t}(\widetilde{\mathcal{Y}}^{j}\,_{i})+(\widetilde{\chi}^{jk})^{m'}\,_{i}(\chi_{m's}),\label{1-cocycle1''}\\
&(\chi_{is})(\widetilde{\chi}^{jk})^{t}=(\mathcal{Y}_{i}\,^{j})^{t}(\widetilde{\mathcal{Y}'}^{k}\,_{s})+
(\mathcal{Y}_{i}\,^{k})^{t}(\widetilde{\mathcal{Y}}^{j}\,_{s})+(\widetilde{\chi}^{jk})^{m'}\,_{s}(\chi_{im'}),\label{1-cocycle2''}\\
&(\chi_{is})(\widetilde{\chi}^{jk})^{t}=(\chi_{is})_{j'}\,^{j}(\widetilde{\chi}^{j'k})^{t}+
(\chi_{is})_{k'}\,^{k}(\widetilde{\chi}^{j'k})^{t}+(\widetilde{\chi}^{jk})^{t}(\chi_{is}),\label{1-cocycle3''}
\end{align}
}
where in the above relations $t$ stands for the transpose of a matrix. On the other hand, \eqref{f3li}-\eqref{t3li}  for the $3$-Leibniz algebra $\mathcal{A}^{\ast}$ in terms of the structure constants as follows:
{\small
\begin{align}
&\widetilde{f}^{ijk}\,_{p}\widetilde{f}^{psm}\,_{n}=\widetilde{f}^{ism}\,_{p}\widetilde{f}^{pjk}\,_{n}+
\widetilde{f}^{jsm}\,_{p}\widetilde{f}^{ipk}\,_{n}+\widetilde{f}^{ksm}\,_{p}\widetilde{f}^{ijp}\,_{n},\label{f3li1'}\\
&\widetilde{f}^{jks}\,_{p}\widetilde{f}^{ipm}\,_{n}=\widetilde{f}^{ijm}\,_{p}\widetilde{f}^{pks}\,_{n}+
\widetilde{f}^{ikm}\,_{p}\widetilde{f}^{jps}\,_{n}+\widetilde{f}^{ism}\,_{p}\widetilde{f}^{jkp}\,_{n},\label{s3li2'}\\
&\widetilde{f}^{ksm}\,_{p}\widetilde{f}^{ijp}\,_{n}=\widetilde{f}^{ijk}\,_{p}\widetilde{f}^{psm}\,_{n}+	\widetilde{f}^{ijs}\,_{p}\widetilde{f}^{kpm}\,_{n}+\widetilde{f}^{ijm}\,_{p}\widetilde{f}^{ksp}\,_{n},\label{t3li3'}
\end{align}
}
where we have the following matrix form of these relations respectively:
{\small
\begin{align}
&(\widetilde{\chi}^{ij})(\widetilde{\mathcal{Y}'}^{s}\,_{n})=(\widetilde{\mathcal{Y}'}^{j}\,_{n})^{t}(\widetilde{\chi}^{is})^{t}+
(\widetilde{\mathcal{Y}}^{i}\,_{n})^{t}(\widetilde{\chi}^{js})^{t}+(\widetilde{\chi}^{ij})^{p}\,_{n}(\widetilde{\mathcal{Y}'}^{s}\,_{p}),
\label{f3li1''}\\
&(\widetilde{\chi}^{jk})(\widetilde{\mathcal{Y}}^{i}\,_{n})=(\widetilde{\mathcal{Y}'}^{k}\,_{n})^{t}(\widetilde{\chi}^{ij})^{t}+
(\widetilde{\mathcal{Y}}^{j}\,_{n})^{t}(\widetilde{\chi}^{ik})^{t}+(\widetilde{\chi}^{jk})^{p}\,_{m}(\widetilde{\mathcal{Y}}^{i}\,_{p}),
\label{s3li2''}\\
&(\widetilde{\chi}^{ks})(\widetilde{\chi}^{ij})=(\widetilde{\chi}^{ij})^{k}\,_{p}(\widetilde{\chi}^{ps})+
(\widetilde{\chi}^{ij})^{s}\,_{p}(\widetilde{\chi}^{kp})+(\widetilde{\chi}^{ij})(\widetilde{\chi}^{ks}).\label{t3li3''}
\end{align}
}
Now, one can use  \eqref{1-cocycle1''}-\eqref{1-cocycle3''} and \eqref{f3li1''}-\eqref{t3li3''} for calculation of the dual $3$-Leibniz algebra $\mathcal{A}^{\ast}$. According to the type of $3$-Leibniz algebras $\mathcal{A}$ and $\mathcal{A}^{\ast}$, we must solve the following equations:
\begin{itemize}
\item
If $\mathcal{A}$ and $\mathcal{A}^{\ast}$ be both the first $3$-Leibniz algebra	
{\small
\begin{align}
\begin{split}
\begin{cases}
(\chi_{is})(\widetilde{\chi}^{jk})^{t}-(\mathcal{Y}'_{s}\,^{j})^{t}(\widetilde{\mathcal{Y}'}^{k}\,_{i})-(\mathcal{Y}'_{s}\,^{k})^{t}(\widetilde{\mathcal{Y}}^{j}\,
_{i})-(\widetilde{\chi}^{jk})^{m'}\,_{i}(\chi_{m's})=0,\\
(\widetilde{\chi}^{ij})(\widetilde{\mathcal{Y}'}^{s}\,_{n})-(\widetilde{\mathcal{Y}'}^{j}\,_{n})^{t}(\widetilde{\chi}^{is})^{t}-
(\widetilde{\mathcal{Y}}^{i}\,_{n})^{t}(\widetilde{\chi}^{js})^{t}-(\widetilde{\chi}^{ij})^{p}\,_{n}(\widetilde{\mathcal{Y}'}^{s}\,_{p})=0.
\label{f-f}
\end{cases}
\end{split}
\end{align}
\item
If $\mathcal{A}$ and $\mathcal{A}^{\ast}$ be the first and the second $3$-Leibniz algebra respectively
\begin{align}
\begin{split}
\begin{cases}
(\chi_{is})(\widetilde{\chi}^{jk})^{t}-(\mathcal{Y}'_{s}\,^{j})^{t}(\widetilde{\mathcal{Y}'}^{k}\,_{i})-(\mathcal{Y}'_{s}\,^{k})^{t}(\widetilde{\mathcal{Y}}^{j}\,_{i})-(\widetilde{\chi}^{jk})^{m'}\,_{i}(\chi_{m's})=0,\\
(\widetilde{\chi}^{jk})(\widetilde{\mathcal{Y}}^{i}\,_{n})-(\widetilde{\mathcal{Y}'}^{k}\,_{n})^{t}(\widetilde{\chi}^{ij})^{t}-
(\widetilde{\mathcal{Y}}^{j}\,_{n})^{t}(\widetilde{\chi}^{ik})^{t}-(\widetilde{\chi}^{jk})^{p}\,_{m}(\widetilde{\mathcal{Y}}^{i}\,_{p})=0.\label{f-s}
\end{cases}
\end{split}
\end{align}
\item
If $\mathcal{A}$ and $\mathcal{A}^{\ast}$ be the first and the third $3$-Leibniz algebra respectively
\begin{align}
\begin{split}
\begin{cases}
(\chi_{is})(\widetilde{\chi}^{jk})^{t}-(\mathcal{Y}'_{s}\,^{j})^{t}(\widetilde{\mathcal{Y}'}^{k}\,_{i})-(\mathcal{Y}'_{s}\,^{k})^{t}(\widetilde{\mathcal{Y}}^{j}\,
_{i})-(\widetilde{\chi}^{jk})^{m'}\,_{i}(\chi_{m's})=0,\\
(\widetilde{\chi}^{ks})(\widetilde{\chi}^{ij})-(\widetilde{\chi}^{ij})^{k}\,_{p}(\widetilde{\chi}^{ps})-
(\widetilde{\chi}^{ij})^{s}\,_{p}(\widetilde{\chi}^{kp})-(\widetilde{\chi}^{ij})(\widetilde{\chi}^{ks})=0.\label{f-t}
\end{cases}
\end{split}
\end{align}
\item
If	$\mathcal{A}$ and $\mathcal{A}^{\ast}$ be the second and the first $3$-Leibniz algebra respectively
\begin{align}
\begin{split}
\begin{cases}
(\chi_{is})(\widetilde{\chi}^{jk})^{t}-(\mathcal{Y}_{i}\,^{j})^{t}(\widetilde{\mathcal{Y}'}^{k}\,_{s})-
(\mathcal{Y}_{i}\,^{k})^{t}(\widetilde{\mathcal{Y}}^{j}\,_{s})-(\widetilde{\chi}^{jk})^{m'}\,_{s}(\chi_{im'})=0,\\
(\widetilde{\chi}^{ij})(\widetilde{\mathcal{Y}'}^{s}\,_{n})-(\widetilde{\mathcal{Y}'}^{j}\,_{n})^{t}(\widetilde{\chi}^{is})^{t}-
(\widetilde{\mathcal{Y}}^{i}\,_{n})^{t}(\widetilde{\chi}^{js})^{t}-(\widetilde{\chi}^{ij})^{p}\,_{n}(\widetilde{\mathcal{Y}'}^{s}\,_{p})=0.
\end{cases}
\end{split}
\end{align}
\item
If	$\mathcal{A}$ and $\mathcal{A}^{\ast}$ be both the second $3$-Leibniz algebra
\begin{align}
\begin{split}
\begin{cases}
(\chi_{is})(\widetilde{\chi}^{jk})^{t}-(\mathcal{Y}_{i}\,^{j})^{t}(\widetilde{\mathcal{Y}'}^{k}\,_{s})-
(\mathcal{Y}_{i}\,^{k})^{t}(\widetilde{\mathcal{Y}}^{j}\,_{s})-(\widetilde{\chi}^{jk})^{m'}\,_{s}(\chi_{im'})=0,\\
(\widetilde{\chi}^{jk})(\widetilde{\mathcal{Y}}^{i}\,_{n})-(\widetilde{\mathcal{Y}'}^{k}\,_{n})^{t}(\widetilde{\chi}^{ij})^{t}-
(\widetilde{\mathcal{Y}}^{j}\,_{n})^{t}(\widetilde{\chi}^{ik})^{t}-(\widetilde{\chi}^{jk})^{p}\,_{m}(\widetilde{\mathcal{Y}}^{i}\,_{p})=0.
\end{cases}
\end{split}
\end{align}
\item
If	$\mathcal{A}$ and $\mathcal{A}^{\ast}$ be the second and the third $3$-Leibniz algebra respectively
\begin{align}
\begin{split}
\begin{cases}
(\chi_{is})(\widetilde{\chi}^{jk})^{t}-(\mathcal{Y}_{i}\,^{j})^{t}(\widetilde{\mathcal{Y}'}^{k}\,_{s})-
(\mathcal{Y}_{i}\,^{k})^{t}(\widetilde{\mathcal{Y}}^{j}\,_{s})-(\widetilde{\chi}^{jk})^{m'}\,_{s}(\chi_{im'})=0,\\
(\widetilde{\chi}^{ks})(\widetilde{\chi}^{ij})-(\widetilde{\chi}^{ij})^{k}\,_{p}(\widetilde{\chi}^{ps})-
(\widetilde{\chi}^{ij})^{s}\,_{p}(\widetilde{\chi}^{kp})-(\widetilde{\chi}^{ij})(\widetilde{\chi}^{ks})=0.
\end{cases}
\end{split}
\end{align}
\item
If	$\mathcal{A}$ and $\mathcal{A}^{\ast}$ be the third and the first $3$-Leibniz algebra respectively
\begin{align}
\begin{split}
\begin{cases}
(\chi_{is})(\widetilde{\chi}^{jk})^{t}-(\chi_{is})_{j'}\,^{j}(\widetilde{\chi}^{j'k})^{t}-
(\chi_{is})_{k'}\,^{k}(\widetilde{\chi}^{j'k})^{t}-(\widetilde{\chi}^{jk})^{t}(\chi_{is})=0,\\
(\widetilde{\chi}^{ij})(\widetilde{\mathcal{Y}'}^{s}\,_{n})-(\widetilde{\mathcal{Y}'}^{j}\,_{n})^{t}(\widetilde{\chi}^{is})^{t}-
(\widetilde{\mathcal{Y}}^{i}\,_{n})^{t}(\widetilde{\chi}^{js})^{t}-(\widetilde{\chi}^{ij})^{p}\,_{n}(\widetilde{\mathcal{Y}'}^{s}\,_{p})=0.
\end{cases}
\end{split}
\end{align}
\item
If	$\mathcal{A}$ and $\mathcal{A}^{\ast}$ be the third and the second $3$-Leibniz algebra respectively
\begin{align}
\begin{split}
\begin{cases}
(\chi_{is})(\widetilde{\chi}^{jk})^{t}-(\chi_{is})_{j'}\,^{j}(\widetilde{\chi}^{j'k})^{t}-
(\chi_{is})_{k'}\,^{k}(\widetilde{\chi}^{j'k})^{t}-(\widetilde{\chi}^{jk})^{t}(\chi_{is})=0,\\
(\widetilde{\chi}^{jk})(\widetilde{\mathcal{Y}}^{i}\,_{n})-(\widetilde{\mathcal{Y}'}^{k}\,_{n})^{t}(\widetilde{\chi}^{ij})^{t}-
(\widetilde{\mathcal{Y}}^{j}\,_{n})^{t}(\widetilde{\chi}^{ik})^{t}-(\widetilde{\chi}^{jk})^{p}\,_{m}(\widetilde{\mathcal{Y}}^{i}\,_{p})=0.
\end{cases}
\end{split}
\end{align}
\item
If	$\mathcal{A}$ and $\mathcal{A}^{\ast}$ be both the third $3$-Leibniz algebra
\begin{align}
\begin{split}
\begin{cases}
(\chi_{is})(\widetilde{\chi}^{jk})^{t}-(\chi_{is})_{j'}\,^{j}(\widetilde{\chi}^{j'k})^{t}-
(\chi_{is})_{k'}\,^{k}(\widetilde{\chi}^{j'k})^{t}-(\widetilde{\chi}^{jk})^{t}(\chi_{is})=0,\\
(\widetilde{\chi}^{ks})(\widetilde{\chi}^{ij})-(\widetilde{\chi}^{ij})^{k}\,_{p}(\widetilde{\chi}^{ps})-
(\widetilde{\chi}^{ij})^{s}\,_{p}(\widetilde{\chi}^{kp})-(\widetilde{\chi}^{ij})(\widetilde{\chi}^{ks})=0.
\end{cases}
\end{split}
\end{align}
}
\end{itemize}
Now, by use of the above relations, we obtain some examples as follows.
\begin{ex}
We consider the following three dimensional first $3$-Leibniz algebras \cite{Zhang}
\begin{enumerate}
\item
$[e_{2},e_{3},e_{3}]=e_{1}\quad,\quad [e_{3},e_{3},e_{3}]=e_{2},$\\
			
By solving the system  of equations \eqref{f-f}- \eqref{f-t} we obtain the following $\mathcal{A}^{\ast}$ algebras:
\begin{itemize}
\item
$\mathcal{A}^{\ast}$ as a	second $3$-Leibniz algebra	
{\small
\begin{align}
[\widetilde{e}^1,\widetilde{e}^1,\widetilde{e}^1]_{\ast}=b\widetilde{e}^2+a\widetilde{e}^3,\quad
[\widetilde{e}^1,\widetilde{e}^2,\widetilde{e}^1]_{\ast}=b\widetilde{e}^3,\nonumber
\end{align}
}
where $a$ and $b$ are any non zero real numbers.
\item
$\mathcal{A}^{\ast}$ as a third $3$-Leibniz algebra
{\small
\begin{align*}
[\widetilde{e}^{1},\widetilde{e}^{1},\widetilde{e}^{1}]_{\ast}=b\widetilde{e}^{2}+a\widetilde{e}^{3},\quad
[\widetilde{e}^{1},\widetilde{e}^{1},\widetilde{e}^{2}]_{\ast}=b\widetilde{e}^{3},
\end{align*}
}
where $a$ and $b$ are any non zero real numbers.
\end{itemize}
\item
$[e_{3},e_{2},e_{3}]=e_{2}\quad,\quad [e_{3},e_{3},e_{2}]=-e_{2},\quad,\quad [e_{3},e_{3},e_{3}]=e_{1}+e_{2},$\\
By solving the  system  of equations \eqref{f-f}- \eqref{f-t} we obtain the following $\mathcal{A}^{\ast}$ algebras:
\begin{itemize}
\item
$\mathcal{A}^{\ast}$ as a first $3$-Leibniz algebra
{\small
\begin{align*}
[\widetilde{e}^{1},\widetilde{e}^{1},\widetilde{e}^{2}]_{\ast}=a\widetilde{e}^{3},\quad
[\widetilde{e}^{1},\widetilde{e}^{2},\widetilde{e}^{2}]_{\ast}=b\widetilde{e}^{3},\quad
[\widetilde{e}^{2},\widetilde{e}^{1},\widetilde{e}^{2}]_{\ast}=c\widetilde{e}^{3},\quad
[\widetilde{e}^{2},\widetilde{e}^{2},\widetilde{e}^{2}]_{\ast}=d\widetilde{e}^{3},
\end{align*}
where $a$, $b$, $c$, $d$ are any non zero real numbers.
\item
$\mathcal{A}^{\ast}$ as second $3$-Leibniz algebra
\begin{align*}
[\widetilde{e}^{1},\widetilde{e}^{1},\widetilde{e}^{1}]_{\ast}=a\widetilde{e}^{3},\quad
[\widetilde{e}^{1},\widetilde{e}^{1},\widetilde{e}^{2}]_{\ast}=b\widetilde{e}^{3},\quad
[\widetilde{e}^{1},\widetilde{e}^{2},\widetilde{e}^{1}]_{\ast}=c\widetilde{e}^{3},\quad
[\widetilde{e}^{1},\widetilde{e}^{2},\widetilde{e}^{2}]_{\ast}=d\widetilde{e}^{3},
\end{align*}
where $a$, $b$, $c$, $d$ are any non zero real numbers.
\item
$\mathcal{A}^{\ast}$ as third $3$-Leibniz algebra
\begin{align*}
& [\widetilde{e}^{1},\widetilde{e}^{1},\widetilde{e}^{1}]_{\ast}=a\widetilde{e}^{3},\quad
[\widetilde{e}^{1},\widetilde{e}^{1},\widetilde{e}^{2}]_{\ast}=b\widetilde{e}^{3},\\
&[\widetilde{e}^{1},\widetilde{e}^{2},\widetilde{e}^{1}]_{\ast}=c\widetilde{e}^{3},\quad
[\widetilde{e}^{1},\widetilde{e}^{2},\widetilde{e}^{2}]_{\ast}=d\widetilde{e}^{3},\\
&[\widetilde{e}^{2},\widetilde{e}^{1},\widetilde{e}^{1}]_{\ast}=m\widetilde{e}^{3},\quad
[\widetilde{e}^{2},\widetilde{e}^{1},\widetilde{e}^{2}]_{\ast}=f\widetilde{e}^{3},\\
&[\widetilde{e}^{2},\widetilde{e}^{2},\widetilde{e}^{1}]_{\ast}=g\widetilde{e}^{3},\quad
[\widetilde{e}^{2},\widetilde{e}^{2},\widetilde{e}^{2}]_{\ast}=h\widetilde{e}^{3},
\end{align*}
}
where $a$, $b$, $c$, $d$, $m$, $f$, $g$, $h$ are any non zero real numbers.
\end{itemize}
\end{enumerate}
\end{ex}	

\section{Correspondence between $3$-Leibniz bialgebra and its associated Leibniz bialgebra}
In this section, we determine the type of the associated Leibniz algebra for any  three types of $3$-Leibniz algebras and prove a theorem about the correspondence
between $3$-Leibniz bialgebra and its associated Leibniz bialgebra \cite{Reza1}.
\begin{itemize}
\item
If $\mathcal{A}$ be the first $3$-Leibniz algebra then its associated Leibniz algebra  $\mathfrak{g}=\mathcal{A}\otimes \mathcal{A}$ with the following bracket
\begin{align}
[x_{1}\otimes x_{2},y_{1}\otimes y_{2}]=[x_{1},y_{1},y_{2}]\otimes x_{2}+x_{1}\otimes [x_{2},y_{1},y_{2}],\label{Leibnizbracket1}
\end{align}
will be a right Leibniz algebra.
\item
If $\mathcal{A}$ be the second $3$-Leibniz algebra then its associated Leibniz algebra $\mathfrak{g}=\mathcal{A}\otimes \mathcal{A}$ with the following brackets
\begin{align}
& [x_{1}\otimes x_{2},y_{1}\otimes y_{2}]=[x_{1},y_{1},x_{2}]\otimes y_{2}+y_{1}\otimes [x_{1},y_{2},x_{2}],\label{Leibnizbracket21}\\
& [x_{1}\otimes x_{2},y_{1}\otimes y_{2}]=[y_{1},x_{1},y_{2}]\otimes x_{2}+x_{1}\otimes [y_{1},x_{2},y_{2}],\label{Leibnizbracket22}
\end{align}
will be a left and a right Leibniz algebra respectively.
\item
If $\mathcal{A}$ be the third $3$-Leibniz algebra then its associated Leibniz algebra $\mathfrak{g}=\mathcal{A}\otimes \mathcal{A}$ with the following bracket
\begin{align}
[x_{1}\otimes x_{2},y_{1}\otimes y_{2}]=[x_{1},x_{2},y_{1}]\otimes y_{2}+y_{1}\otimes [x_{1},x_{2},y_{2}],\label{Leibnizbracket3}
\end{align}
will be a left Leibniz algebra.
\end{itemize}
Now we prove a theorem about the correspondence between $3$-Leibniz bialgebra and its associated Leibniz bialgebra \cite{Reza1}.
\begin{thm}\label{corresponding}
{\textit{Let $(\mathcal{A},\gamma)$ be a $3$-Leibniz bialgebra and $\mathfrak{g}=\mathcal{A}\otimes\mathcal{A}$ be its associated Leibniz algebra then there exist a linear map
$\delta:\mathfrak{g}\longrightarrow\mathfrak{g}\otimes\mathfrak{g}$ which it defines a Leibniz bialgebra structure on $\mathfrak{g}$. Conversely, if $(\mathfrak{g},\delta)$ be a Leibniz bialgebra such that $\mathfrak{g}=\mathcal{A}\otimes\mathcal{A}$ and $\mathcal{A}$ be a $3$-Leibniz algebra then there exist a
linear map $\gamma:\mathcal{A}\longrightarrow \mathcal{A}\otimes \mathcal{A}\otimes\mathcal{A}$ such that it defines a $3$-Leibniz bialgebra structure on $\mathcal{A}$.}}
\end{thm}
\begin{proof} Since $\mathcal{A}$ can be a $3$-Leibniz algebra of various three types the proof of the theorem divided to the following three parts\footnote{Here we write only the proof of one case. The proof of the other cases is similar.}.
\begin{enumerate}	
\item
If $(\mathcal{A},\gamma)$ be a $3$-Leibniz bialgebra such that $(\ref{1-cocycle1})$ is valid for $\gamma[x_{1},x_{2},x_{3}]$ then $\mathfrak{g}=\mathcal{A}\otimes\mathcal{A}$
will be a right Leibniz algebra with  the bracket \eqref{Leibnizbracket1}. $\mathcal{A}^{\ast}$ can be a $3$-Leibniz algebra of various three types.
\begin{enumerate}
\item
If $\mathcal{A}^{\ast}$ be the first $3$-Leibniz algebra then  $\mathfrak{g}^{\ast}=\mathcal{A}^{\ast}\otimes\mathcal{A}^{\ast}$ with the bracket \eqref{Leibnizbracket1} will be a right Leibniz algebra. We want to prove there exist a linear map $\delta:\mathfrak{g}\longrightarrow\mathfrak{g}\otimes\mathfrak{g}$ such that it defines a Leibniz bialgebra structure on $\mathfrak{g}$. We can rewrite the bracket $[.,.]_{\ast}$ on $\mathfrak{g}^{\ast}$ with $\gamma^{t}$ as follows:
\begin{align}
[\widetilde{x}^{1}\otimes\widetilde{x}^{2},\widetilde{y}^{1}\otimes\widetilde{y}^{2}]_{\ast}=
\gamma^{t}(\widetilde{x}^{1}\otimes\widetilde{y}^{1}\otimes\widetilde{y}^{2})\otimes \widetilde{x}^{2}+\widetilde{x}^{1}\otimes\gamma^{t}(\widetilde{x}^{2}\otimes \widetilde{y}^{1}\otimes\widetilde{y}^{2}),\label{bracket*}
\end{align}
using the following flip operators
\begin{align}
&\sigma_{24}:{\mathcal{A}^{\ast}}^{\otimes 4}\longrightarrow {\mathcal{A}^{\ast}}^{\otimes 4}
,\,\,\sigma_{24}(\xi_{1}\otimes\xi_{2}\otimes\xi_{3}\otimes\xi_{4})=\xi_{1}\otimes\xi_{4}\otimes\xi_{3}\otimes\xi_{2}\\
&\sigma_{34}:{\mathcal{A}^{\ast}}^{\otimes 4}\longrightarrow {\mathcal{A}^{\ast}}^{\otimes 4}
,\,\,\sigma_{34}(\xi_{1}\otimes\xi_{2}\otimes\xi_{3}\otimes\xi_{4})=\xi_{1}\otimes\xi_{2}\otimes\xi_{4}\otimes\xi_{3},
\end{align}
\eqref{bracket*} can be written as
\begin{align}
[\widetilde{x}^{1}\otimes\widetilde{x}^{2},\widetilde{y}^{1}\otimes\widetilde{y}^{2}]_{\ast}=\left((\gamma^{t}\otimes I_{\cal{A}^{\ast}})\circ\sigma_{24}\circ\sigma_{34}+I_{\cal{A}^{\ast}}\otimes\gamma^{t}\right)(\widetilde{x}^{1}\otimes\widetilde{x}^{2}\otimes\widetilde{y}^{1}\otimes\widetilde{y}^{2}),
\end{align}
setting
\begin{align}
\delta^{t}:=[,]_{\ast},
\end{align}
then we have
\begin{align}
\delta^{t}=(\gamma^{t}\otimes I_{\mathcal{A}^{\ast}})\circ\sigma_{24}\circ\sigma_{34}+I_{\mathcal{A}^{\ast}}\otimes\gamma^{t},
\end{align}
and so
\begin{align}
\delta(x_1\otimes x_2)=\left(\sigma_{34}^{t}\circ\sigma_{24}^{t}\circ(\gamma\otimes I_\mathcal{A})+I_\mathcal{A}\otimes \gamma\right)(x_1\otimes x_2),
\end{align}
where $\sigma_{24}^{t},\sigma_{34}^{t}:\mathcal{A}^{\otimes 4}\longrightarrow \mathcal{A}^{\otimes 4}$  act as follows
\begin{align}
\sigma_{24}^{t}(x_1\otimes x_2\otimes x_3\otimes x_4)=x_1\otimes x_4\otimes x_3\otimes x_2,\\
\sigma_{34}^{t}(x_1\otimes x_2\otimes x_3\otimes x_4)=x_1\otimes x_2\otimes x_4\otimes x_3,
\end{align}
then for any $X,Y\in \mathfrak{g}$ we have
\begin{align}
\begin{split}
&\delta[X,Y]=\delta[x_1\otimes x_2,y_1\otimes y_2]=\delta([x_1,y_1,y_2]\otimes x_2+x_1\otimes [x_2,y_1,y_2])\\
&=\left(\sigma_{34}^{t}\circ\sigma_{24}^{t}\circ(\gamma\otimes I_\mathcal{A})+I_\mathcal{A}\otimes \gamma\right)([x_1,y_1,y_2]\otimes x_2+x_1\otimes [x_2,y_1,y_2])\\
& =\left(1_\mathfrak{g}\otimes \mathrm{ad}_Y^{(r)}+\mathrm{ad}_Y^{(r)}\otimes1_\mathfrak{g} \right)\delta(X)
\end{split}
\end{align}
where in  the above identity we use $\gamma(x_{1})=x_{1}^{1}\otimes x_{1}^{2}\otimes x_{1}^{3}$ and $\gamma(x_{2})=x_{2}^{1}\otimes x_{2}^{2}\otimes x_{2}^{3}$.\\
In the same way one can prove the following cases:
\item
If $\mathcal{A}^{\ast}$ be the second $3$-Leibniz algebra then
\begin{enumerate}
\item
If $\mathfrak{g}^{\ast}=\mathcal{A}^{\ast}\otimes\mathcal{A}^{\ast}$  be a left Leibniz algebra with  the bracket \eqref{Leibnizbracket21} then
\begin{align}
\delta(x_1\otimes x_2)=\left(\sigma_{23}^t\circ(\gamma\otimes I_\mathcal{A})+\sigma_{34}^t\circ\sigma_{24}^t\circ\sigma_{12}^t\circ(I_\mathcal{A}\otimes\gamma)\right)(x_1\otimes x_2).
\end{align}
\item
If $\mathfrak{g}^{\ast}=\mathcal{A}^{\ast}\otimes\mathcal{A}^{\ast}$  be a right Leibniz algebra with the bracket  \eqref{Leibnizbracket22} then
\begin{align}
\delta(x_1\otimes x_2)=\left(\sigma_{34}^t\circ\sigma_{24}^t\circ\sigma_{12}^t\circ(\gamma\otimes I_\mathcal{A})+\sigma_{23}^t\circ(I_\mathcal{A}\otimes\gamma)\right)(x_1\otimes x_2).
\end{align}
\end{enumerate}
\item
If $\mathcal{A}^{\ast}$ be the third $3$-Leibniz algebra then $\mathfrak{g}^{\ast}=\mathcal{A}^{\ast}\otimes\mathcal{A}^{\ast}$  will be a left Leibniz algebra with the bracket  \eqref{Leibnizbracket3}. We have
\begin{align}
\delta(x_1\otimes x_2)=\left(\gamma\otimes I_\mathcal{A}+\sigma_{23}^t\circ \sigma_{23}^t\circ(I_\mathcal{A}\otimes\gamma)\right)(x_1\otimes x_2).
\end{align}
In all above cases $1$-cocycle condition is
\begin{align}
\delta[X,Y]=\left(1_\mathfrak{g}\otimes \mathrm{ad}_Y^{(r)}+\mathrm{ad}_Y^{(r)}\otimes1_\mathfrak{g} \right)\delta(X),\,\,\, \forall X,Y\in\mathfrak{g}.
\end{align}
\end{enumerate}
\item
If $(\mathcal{A},\gamma)$ be a $3$-Leibniz bialgebra such that \eqref{1-cocycle2} is valid for $\gamma[x_{1},x_{2},x_{3}]$ then we will have the following cases:
\begin{enumerate}
\item
If $\mathcal{A}^{\ast}$ be the first $3$-Leibniz algebra then $\mathfrak{g}^{\ast}=\mathcal{A}^{\ast}\otimes\mathcal{A}^{\ast}$  will be a right Leibniz algebra with  the bracket  \eqref{Leibnizbracket1}.
\begin{enumerate}
\item
If $\mathfrak{g}=\mathcal{A}\otimes \mathcal{A}$ be a left Leibniz algebra with the bracket
\eqref{Leibnizbracket21} then we will have
\begin{align*}
\delta[X,Y]=\left(1_\mathfrak{g}\otimes \mathrm{ad}_X^{(l)}+\mathrm{ad}_X^{(l)}\otimes1_\mathfrak{g} \right)\delta(Y),\,\,\,\forall X,Y\in\mathfrak{g}.
\end{align*}
\item
If $\mathfrak{g}=\mathcal{A}\otimes \mathcal{A}$ be a right Leibniz algebra with the bracket
\eqref{Leibnizbracket22} then we will have
\begin{align*}
\delta[X,Y]=\left(1_\mathfrak{g}\otimes \mathrm{ad}_Y^{(r)}+\mathrm{ad}_Y^{(r)}\otimes1_\mathfrak{g} \right)\delta(X),\,\,\,\forall X,Y\in\mathfrak{g},
\end{align*}
in the above two cases we have
\begin{align}
\delta(x_1\otimes x_2)=\left(\sigma_{34}^t\circ\sigma_{24}^t\circ(\gamma\otimes I_\mathcal{A})+I_\mathcal{A}\otimes\gamma\right)(x_1\otimes x_2).
\end{align}
\end{enumerate}
\item
If $\mathcal{A}^{\ast}$ be the second $3$-Leibniz algebra then
\begin{enumerate}
\item
If $\mathfrak{g}=\mathcal{A}\otimes \mathcal{A}$ and $\mathfrak{g}^{\ast}=\mathcal{A}^{\ast}\otimes\mathcal{A}^{\ast}$ are both a left Leibniz algebra with  the bracket \eqref{Leibnizbracket21} then we will have
\begin{align*}
\delta[X,Y]=\left(1_\mathfrak{g}\otimes \mathrm{ad}_X^{(l)}+\mathrm{ad}_X^{(l)}\otimes1_\mathfrak{g} \right)\delta(Y),\,\,\,\forall X,Y\in\mathfrak{g}.
\end{align*}
\item
If $\mathfrak{g}=\mathcal{A}\otimes \mathcal{A}$ be a right Leibniz algebra with the bracket \eqref{Leibnizbracket22} and $\mathfrak{g}^{\ast}=\mathcal{A}^{\ast}\otimes\mathcal{A}^{\ast}$ be a left Leibniz algebra with the bracket  \eqref{Leibnizbracket21} then we will have
\begin{align*}
\delta[X,Y]=\left(1_\mathfrak{g}\otimes \mathrm{ad}_Y^{(r)}+\mathrm{ad}_Y^{(r)}\otimes1_\mathfrak{g} \right)\delta(X),\,\,\,\forall X,Y\in\mathfrak{g},
\end{align*}
where in the above two cases we have
\begin{align}
\delta(x_1\otimes x_2)=\left(\sigma_{23}^t\circ(\gamma\otimes I_\mathcal{A})+\sigma_{34}^t\circ\sigma_{24}^t\circ\sigma_{12}^t\circ(I_\mathcal{A}\otimes\gamma)\right)(x_1\otimes x_2).
\end{align}	
\item
If $\mathfrak{g}=\mathcal{A}\otimes \mathcal{A}$ be a left Leibniz algebra with the bracket \eqref{Leibnizbracket21} and $\mathfrak{g}^{\ast}=\mathcal{A}^{\ast}\otimes\mathcal{A}^{\ast}$ be a right Leibniz algebra with the bracket  \eqref{Leibnizbracket22} then we will have
\begin{align*}
\delta[X,Y]=\left(1_\mathfrak{g}\otimes \mathrm{ad}_X^{(l)}+\mathrm{ad}_X^{(l)}\otimes1_\mathfrak{g} \right)\delta(Y),\,\,\,\forall X,Y\in\mathfrak{g}
\end{align*}
\item
If $\mathfrak{g}=\mathcal{A}\otimes \mathcal{A}$ and $\mathfrak{g}^{\ast}=\mathcal{A}^{\ast}\otimes\mathcal{A}^{\ast}$ are both a right Leibniz algebra with the bracket  \eqref{Leibnizbracket22} then we will have
\begin{align*}
\delta[X,Y]=\left(1_\mathfrak{g}\otimes \mathrm{ad}_Y^{(r)}+\mathrm{ad}_Y^{(r)}\otimes1_\mathfrak{g} \right)\delta(X),\,\,\,\forall X,Y\in\mathfrak{g},
\end{align*}
where in the above two cases we have
\begin{align}
\delta(x_1\otimes x_2)=\left(\sigma_{34}^t\circ\sigma_{24}^t\circ\sigma_{12}^t
\circ(\gamma\otimes \sigma_{23}^t\circ I_\mathcal{A})+\circ(I_\mathcal{A}\otimes\gamma)\right)(x_1\otimes x_2).
\end{align}
\end{enumerate}
\item
If $\mathcal{A}^{\ast}$ be the third $3$-Leibniz algebra then we will have:
\begin{enumerate}
\item
If	$\mathfrak{g}=\mathcal{A}\otimes \mathcal{A}$ be a left Leibniz algebra with  the bracket \eqref{Leibnizbracket21} and $\mathfrak{g}^{\ast}=\mathcal{A}^{\ast}\otimes\mathcal{A}^{\ast}$ be a left Leibniz algebra with  the bracket \eqref{Leibnizbracket3} then we will have
\begin{align*}
\delta[X,Y]=\left(1_\mathfrak{g}\otimes \mathrm{ad}_X^{(l)}+\mathrm{ad}_X^{(l)}\otimes1_\mathfrak{g} \right)\delta(Y),\,\,\,\forall X,Y\in\mathfrak{g}.
\end{align*}
\item
If $\mathfrak{g}=\mathcal{A}\otimes \mathcal{A}$ be a right Leibniz algebra with the bracket \eqref{Leibnizbracket22} and $\mathfrak{g}^{\ast}=\mathcal{A}^{\ast}\otimes\mathcal{A}^{\ast}$ be a left Leibniz algebra with the bracket  \eqref{Leibnizbracket3}  then we will have
\begin{align*}
\delta[X,Y]=\left(1_\mathfrak{g}\otimes \mathrm{ad}_Y^{(r)}+\mathrm{ad}_Y^{(r)}\otimes1_\mathfrak{g} \right)\delta(X),\,\,\,\forall X,Y\in\mathfrak{g},
\end{align*}
where in the above two cases we have
\begin{align}
\delta(x_1\otimes x_2)=\left(\gamma\otimes I_\mathcal{A} +\sigma_{23}^t\circ\sigma_{12}^t\circ(I_\mathcal{A}\otimes\gamma)\right)(x _1\otimes x_2).
\end{align}	
\end{enumerate}		 	
\end{enumerate}
\item
If $(\mathcal{A},\gamma)$ be a $3$-Leibniz bialgebra such that $(\ref{1-cocycle3})$ is valid for $\gamma[x_{1},x_{2},x_{3}]$ then $\mathfrak{g}=\mathcal{A}\otimes\mathcal{A}$ with the bracket \eqref{Leibnizbracket3} is a left Leibniz algebra and also $\mathcal{A}^{\ast}$ can be  a $3$-Leibniz algebra of various three types
\begin{enumerate}
\item
If $\mathcal{A}^{\ast}$ be the first $3$-Leibniz algebra then   $\mathfrak{g}^{\ast}=\mathcal{A}^{\ast}\otimes\mathcal{A}^{\ast}$ with the bracket  \eqref{Leibnizbracket1} is a right Leibniz algebra then we will have
\begin{align}
\delta(x_1\otimes x_2)=\left(\sigma_{34}^t\circ\sigma_{24}^t\circ(\gamma\otimes I_\mathcal{A}) +I_\mathcal{A}\otimes\gamma\right)(x _1\otimes x_2).
\end{align}	
\item
If $\mathcal{A}^{\ast}$ be the second $3$-Leibniz algebra then we will have
\begin{enumerate}
\item
If $\mathfrak{g}=\mathcal{A}\otimes \mathcal{A}$ be a left Leibniz algebra with the bracket \eqref{Leibnizbracket21} then we will have
\begin{align}
\delta(x_1\otimes x_2)=\left(\sigma_{23}^t\circ(\gamma\otimes I_\mathcal{A}) +\sigma_{34}^t\circ\sigma_{24}^t\circ\sigma_{12}^t\circ(I_\mathcal{A}\otimes\gamma)\right)(x _1\otimes x_2).
\end{align}
\item
If $\mathfrak{g}=\mathcal{A}\otimes \mathcal{A}$ be a right Leibniz algebra with the bracket \eqref{Leibnizbracket22} then we will have
\begin{align}
\delta(x_1\otimes x_2)=\left(\sigma_{34}^t\circ\sigma_{24}^t\circ\sigma_{12}^t\circ(\gamma\otimes I_\mathcal{A}) +\sigma_{23}^t\circ(I_\mathcal{A}\otimes\gamma)\right)(x _1\otimes x_2).
\end{align}
\end{enumerate}	
\item
If $\mathcal{A}^{\ast}$ be the second $3$-Leibniz algebra then $\mathfrak{g}^{\ast}=\mathcal{A}^{\ast}\otimes\mathcal{A}^{\ast}$ with  the bracket \eqref{Leibnizbracket3} will be a left Leibniz algebra. We have
\begin{align}
\delta(x_1\otimes x_2)=\left(\gamma\otimes I_\mathcal{A} +\sigma_{23}^t\circ\sigma_{12}^t\circ(I_\mathcal{A}\otimes\gamma)\right)(x _1\otimes x_2),
\end{align}
where in  the above cases we have
\begin{align*}
\delta[X,Y]=\left(1_\mathfrak{g}\otimes \mathrm{ad}_X^{(l)}+\mathrm{ad}_X^{(l)}\otimes1_\mathfrak{g} \right)\delta(Y),\,\,\,\forall X,Y\in\mathfrak{g}.
\end{align*}
\end{enumerate}
\end{enumerate}
The proof of the inverse is clear.
\end{proof} 
	
\section{$3$-Lie bialgebras}
In this section, we suppose that $\mathcal{A}$ is a $3$-Lie algebra as a special case, then we have the following fundamental identity for $n=3$
\begin{align}
[x_1, x_2,[y_1,y_2,y_3]]=[[x_1, x_2,y_1],y_2,y_3]+[y_1,[x_1, x_2,y_2],y_3]
+[y_1,y_2,[x_1, x_2,y_3]],\label{fun3}
\end{align}
We want to define a bialgebra structure for $3$-Lie algebra, similar to $3$-Leibniz algebra with a little difference.
\begin{rem}
In Definition \ref{rep3l} if $\rho$ satisfies only in  the identity \eqref{c1} we say $\rho$ is a {\textit semi-representation} of $\mathcal{A}$ in $V$.	
\end{rem}
If $V=\mathcal{A}$ it is clearly that $\rho:\mathcal{A}\otimes \mathcal{A}\longrightarrow End(\mathcal{A})$ with the following relation
\begin{align*}
\rho(x_1, x_2)(z)=\mathrm{ad}_{(x_1, x_2,\hat{x_3})}(z)=[x_1, x_2,z],
\end{align*}
is a representation of $\mathcal{A}$ in $\mathcal{A}$.
\begin{rem}
If $\rho_i:\mathcal{A}\otimes \mathcal{A}\longrightarrow End(V_i),i=1,2,3$ be three representations of $\mathcal{A}$ in vector spaces $V_i$ then $\rho:\mathcal{A}\otimes \mathcal{A}\longrightarrow End(V_1\otimes V_2\otimes V_3)$ with the following identities
\begin{align*}
\rho(x_1, x_2)(y_1\otimes y_2\otimes y_3)&=\rho_1(x_1, x_2)(y_1)\otimes y_2\otimes y_3+y_1\otimes\rho_2(x_1, x_2)(y_2)\otimes y_3\\
&+y_1\otimes y_2\otimes\rho_3(x_1, x_2)(y_3),
\end{align*}
$\forall x_1, x_2\in \mathcal{A},\forall y_i\in V_i,i=1,2,3$ will not be a representation  of $\mathcal{A}$ in $V_1\otimes V_2\otimes V_3$ but it is a semi-representation of $\mathcal{A}$ in $V_1\otimes V_2\otimes V_3$. Note that
in Lie algebra case the tensor product of representations of Lie algebra in vector spaces $V_i,i=1, . . ., n$ is a representation of the Lie algebra in $V_1\otimes V_2\otimes...\otimes V_n$. If $V_1=V_2=V_3=\mathcal{A}$ then $\rho:\mathcal{A}\otimes \mathcal{A}\longrightarrow End(\mathcal{A}^{\otimes 3})$ with the following relation
{\small
\begin{align}
\begin{split}
\rho(y_1,y_2)(z_1\otimes z_2\otimes z_3)&=(\mathrm{ad}_{(y_1,y_2,\hat{y}_3)}\otimes 1\otimes 1+1\otimes \mathrm{ad}_{(y_1,y_2,\hat{y}_3)}\otimes 1+1\otimes 1\otimes \mathrm{ad}_{(y_1,y_2,\hat{y}_3)})(z_1\otimes z_2\otimes z_3)\\
&:={ad}^{(3)}_{(y_1,y_2,\hat{y}_3)}(z_1\otimes z_2\otimes z_3),\label{semi-rep}
\end{split}
\end{align}
}
will be a semi-representation of $\mathcal{A}$ in $\mathcal{A}^{\otimes 3}$.	
\end{rem}
We need to generalize the definition \ref{cohol} for any representation of $\mathcal{A}$ in any vector space $V$.
\begin{dfn}
Let $\mathcal{A}$ be a $3$-Lie algebra, a $V$-valued $p$-cochain is a linear map \linebreak $\psi:(\mathcal{A}\otimes \mathcal{A})^{\otimes(p-1)}\otimes \mathcal{A}\longrightarrow V$. We denote the space of $V$-valued $p$-cochains with $\Gamma^{p}(\mathcal{A},V)$, the coboundary operator is given by:
\begin{align*}
d^{p}\psi(x_1, . . . ,x_{2p+1})&=\sum_{j=1}^{p}\sum_{k=2j+1}^{2p+1}(-1)^{j}\psi(x_1, . . ., \hat{x}_{j-1},\hat{x}_j, . . ., [x_{2j-1},x_{2j},x_k], . . ., x_{2p+1})\\
&\;\;+\sum_{k=1}^{p}\rho(x_{2k-1},x_{2k},\psi(x_1, . . ., \hat{x}_{2k-1},\hat{x}_{2k}, . . ., x_{2p+1}))\nonumber\\
&\;\;-(-1)^{p+1}\rho(x_{2p-1},x_{2p+1},\psi(x_1, . . ., x_{2p-2},x_{2p}),)\\
&\;\;+(-1)^{p+1}\rho(x_{2p},x_{2p+1},\psi(x_1, . . ., x_{2p-1})),
\end{align*}
where $\rho:\mathcal{A}\otimes \mathcal{A}\longrightarrow End(V)$ is a representation of $\mathcal{A}$ in $V$.\\
For $p=1$ we have
\begin{align*}
d^{1}\psi(x_1, x_2, x_3)=-\psi([x_1, x_2, x_3])+\rho(x_1, x_2,\psi(x_3))-\rho(x_1, x_3,\psi(x_2))+\rho(x_2, x_3,\psi(x_1)).
\end{align*}
\end{dfn}
\begin{dfn}
A $3$-Lie bialgebra $(\mathcal{A},\gamma)$ is a $3$-Lie algebra $\mathcal{A}$ with a linear map (cocommutator) $\gamma:\mathcal{A}\longrightarrow \mathcal{A}^{\otimes 3}$ such that
\begin{itemize}
\item
$\gamma$ is a $1$-cocycle on $\mathcal{A}$ with values in $\mathcal{A}^{\otimes 3}$,
\begin{align}
\gamma[y_1,y_2,y_3]=\rho(y_2,y_3,\gamma(y_1))+\rho(y_3,y_1,\gamma(y_2))+\rho(y_1,y_2,\gamma(y_3)),\label{1-cocycle3lie}
\end{align}
where $\rho$ is a semi-representation of $\mathcal{A}$ in
$\mathcal{A}^{\otimes 3}$.
\item
$\gamma^{t}:{\mathcal{A}^{\ast}}^{\otimes 3}\longrightarrow \mathcal{A}$ defines a $3$-Lie bracket on $\mathcal{A}^{\ast}$.		
\end{itemize}
If we use \eqref{notation} then we will have  \eqref{pairing} similar to $3$-Leibniz bialgebra.	
\end{dfn}
By the use of \eqref{semi-rep} $1$-cocycle condition \eqref{1-cocycle3lie} can be rewritten as follows
\begin{align}
\gamma[y_1,y_2,y_3]={ad}^{(3)}_{(y_2,y_3,\hat{y}_1)}\gamma(y_1)
+{ad}^{(3)}_{(y_3,y_1,\hat{y}_2)}\gamma(y_2)
+{ad}^{(3)}_{(y_1,y_2,\hat{y}_3)}\gamma(y_3).\label{1-cocycle3l}
\end{align}
\begin{prop}
If $(\mathcal{A},\gamma)$ be a $3$-Lie bialgebra, and $\mu$ be a $3$-Lie bracket of $\mathcal{A}$, then $(\mathcal{A},\mu^{t})$ will be a $3$-Lie bialgebra,
where $\gamma^{t}$ is a $3$-Lie bracket of $\mathcal{A}^{\ast}$.
\end{prop}
\begin{proof}
By  the use of \eqref{pairing} and \eqref{1-cocycle3l} the proof is clear.
\end{proof}
\begin{thm}
{\textit{Let $(\mathcal{A},\gamma)$ be a $3$-Lie bialgebra and $\mathfrak{g}=\mathcal{A}\otimes\mathcal{A}$ be its associated Leibniz algebra then there exist a linear map
$\delta:\mathfrak{g}\longrightarrow\mathfrak{g}\otimes\mathfrak{g}$ where it defines a Leibniz bialgebra structure on $\mathfrak{g}$. Conversely, if $(\mathfrak{g},\delta)$ is a Leibniz bialgebra such that $\mathcal{A}$ be a $3$-Lie algebra and $\mathfrak{g}=\mathcal{A}\otimes\mathcal{A}$ then there exist a
linear map $\gamma:\mathcal{A}\longrightarrow \mathcal{A}\otimes \mathcal{A}\otimes\mathcal{A}$ such that it defines a $3$-Lie bialgebra structure on $\mathcal{A}$.}}
\end{thm}
\begin{proof}
From  theorem \ref{corresponding} the proof is clear.
\end{proof}	
\subsection{$3$-Lie bialgebra in terms of structure constants; some examples}
	
Here we rewrite the $1$-cocycle condition \eqref{1-cocycle3l} in terms of structure constants of $3$-Lie algebras $\mathcal{A}$ and $\mathcal{A}^{\ast}$.
Note that in this case we have \eqref{st-con} and \eqref{cocommutator} same as $3$-Leibniz algebra with a little difference. In this case the structure constants are antisymmetric.
\begin{align}
\gamma[x_i, x_s,x_n]={ad}^{(3)}_{(x_s,x_n,\hat{x}_i)}\gamma(x_i)
+{ad}^{(3)}_{(x_n,x_i,\hat{x}_s)}\gamma(x_s)
+{ad}^{(3)}_{(x_i, x_s,\hat{x}_n)}\gamma(x_n).\label{1-cocycle3ls}
\end{align}
Using  \eqref{pairing} and \eqref{st-con} in \eqref{1-cocycle3ls}  we have
\begin{align}
\begin{split}
f_{isn}\,^{p}\widetilde{f}^{jkm}\,_{p}&=
f_{j'sn}\,^{j}\widetilde{f}^{j'km}\,_{i}
+f_{k'sn}\,^{k}\widetilde{f}^{jk'm}\,_{i}
+f_{m'sn}\,^{m}\widetilde{f}^{jkm'}\,_{i}\\
&\;\; +f_{ij'n}\,^{j}\widetilde{f}^{j'km}\,_{s}
+f_{ik'n}\,^{k}\widetilde{f}^{jk'm}\,_{s}
+f_{im'n}\,^{m}\widetilde{f}^{jkm'}\,_{s}\\
&\;\; +f_{isj'}\,^{j}\widetilde{f}^{j'km}\,_{n}
+f_{isk'}\,^{k}\widetilde{f}^{jk'm}\,_{n}
+f_{ism'}\,^{m}\widetilde{f}^{jkm'}\,_{n}.\label{1-cocycle3lconstant}
\end{split}
\end{align}
\begin{dfn}
If the structure constants of the $3$-Lie algebras $\mathcal{A}$ and $\mathcal{A}^{\ast}$ satisfy in relation \eqref {1-cocycle3lconstant} then $\mathcal{A}$ and $\mathcal{A}^{\ast}$ will construct a $3$-Lie bialgebra.
\end{dfn}
To use this relation in  calculations, one can rewrite it in terms of the matrix form  by  use of the following adjoint representations:
\begin{align}
\begin{split}
& f_{isn}\,^{p}=(\chi_{is})_{n}\,^{p}=
(\mathcal{Y}_{i}\,^{p})_{sn}=f_{sni}\,^{p}=(\chi_{sn})_{i}\,^{p}=(\mathcal{Y}_{s}\,^{p})_{ni}\\
& =f_{nis}\,^{p}=(\chi_{ni})_{s}\,^{p}=
(\mathcal{Y}_{n}\,^{p})_{is}=-f_{sin}\,^{p}=-(\chi_{si})_{n}\,^{p}=-(\mathcal{Y}_{s}\,^{p})_{in}\\
& =-f_{nsi}\,^{p}=-(\chi_{ns})_{i}\,^{p}=
-(\mathcal{Y}_{n}\,^{p})_{si}
=f_{ins}\,^{p}=-(\chi_{in})_{s}\,^{p}=
-(\mathcal{Y}_{i}\,^{p})_{ns},
\end{split}
\end{align}
\begin{align}
\begin{split}
& \widetilde{f}^{jkm}\,_{p}=(\widetilde{\chi}^{jk})^{m}
\,_{p}=(\widetilde{\mathcal{Y}}^{j}\,_{p})^{km}
=\widetilde{f}^{kmj}\,_{p}=(\widetilde{\chi}^{km})^{j}
\,_{p}=(\widetilde{\mathcal{Y}}^{k}\,_{p})^{mj}\\
& =\widetilde{f}^{mjk}\,_{p}=(\widetilde{\chi}^{mj})^{k}
\,_{p}=(\widetilde{\mathcal{Y}}^{m}\,_{p})^{jk}
=-\widetilde{f}^{kjm}\,_{p}=-(\widetilde{\chi}^{kj})^{m}
\,_{p}=-(\widetilde{\mathcal{Y}}^{k}\,_{p})^{jm}\\
& =-\widetilde{f}^{mkj}\,_{p}=-(\widetilde{\chi}^{mk})^{j}
\,_{p}=-(\widetilde{\mathcal{Y}}^{m}\,_{p})^{kj}
=-\widetilde{f}^{jmk}\,_{p}=-(\widetilde{\chi}^{jm})^{k}
\,_{p}=-(\widetilde{\mathcal{Y}}^{j}\,_{p})^{mk}.
\end{split}
\end{align}
 Relation \eqref {1-cocycle3lconstant} has the following matrix form
\begin{align}
\begin{split}
\chi_{ns}(\widetilde{\chi}^{mk})^{t}&=(\widetilde{\chi}^{mk})^{t}(\chi_{ns})-(\chi_{k's})_{n}\,^{k}(\widetilde{\chi}^{mk'})^{t}-(\widetilde{\chi}^{m'k})^{t}(\chi_{m's})_{n}\,^{m}\\
& -(\widetilde{\chi}^{j'k})^{m}\,_{s}(\chi_{nj'})+\mathcal{Y}_{n}\,^{k}
(\widetilde{\mathcal{Y}}^{m}\,_{s})^{t}+\mathcal{Y}_{n}\,^{p}\widetilde{\mathcal{Y}}^{k}\,_{s}\\
& -(\widetilde{\chi}^{j'k})^{m}\,_{n}\chi_{j's}
-\mathcal{Y}_{s}\,^{k}(\widetilde{\mathcal{Y}}^{m}\,_{n})^{t}+\mathcal{Y}_{s}\,^{m}(\widetilde{\mathcal{Y}}^{k}\,_{s})^{t},
\end{split}
\end{align}
where in the above relation $t$ stands for  the transpose of a matrix.
Fundamental identity for $3$-Lie algebra $\mathcal{A}^{\ast}$ in terms of the structure constant has the  following form
\begin{align}
\widetilde{f}^{nsj}\,_p\widetilde{f}^{mkp}\,_i=
\widetilde{f}^{mkn}\,_p\widetilde{f}^{psj}\,_i
+\widetilde{f}^{mks}\,_p\widetilde{f}^{npj}\,_i
+\widetilde{f}^{mkj}\,_p\widetilde{f}^{nsp}\,_i,
\end{align}
with the matrix form as
\begin{align}
\widetilde{\chi}^{ns}\widetilde{\chi}^{mk}=(\widetilde{\chi}^{mk})^{n}\,_p
\widetilde{\chi}^{ps}+(\widetilde{\chi}^{mk})^{s}\,_p\widetilde{\chi}^{np}
+\widetilde{\chi}^{mk}\widetilde{\chi}^{ns}.\label{matdual}
\end{align}
Now for obtaining the dual of $\mathcal{A}$ one must solve the following equation with \eqref{matdual}.
\begin{align}
\begin{split}
&(\widetilde{\chi}^{mk})^{t}(\chi_{ns})-(\chi_{k's})_{n}\,^{k}(\widetilde{\chi}^{mk'})^{t}-(\widetilde{\chi}^{m'k})^{t}(\chi_{m's})_{n}\,^{m}
-(\widetilde{\chi}^{j'k})^{m}\,_{s}(\chi_{nj'})+\mathcal{Y}_{n}\,^{k}(\widetilde{\mathcal{Y}}^{m}\,_{s})^{t}+\mathcal{Y}_{n}\,^{p}\widetilde{\mathcal{Y}}^{k}\,_{s}\\
& -(\widetilde{\chi}^{j'k})^{m}\,_{n}\chi_{j's}
-\mathcal{Y}_{s}\,^{k}(\widetilde{\mathcal{Y}}^{m}\,_{n})^{t}+\mathcal{Y}_{s}\,^{m}(\widetilde{\mathcal{Y}}^{k}\,_{s})^{t}-\chi_{ns}(\widetilde{\chi}^{mk})^{t}=0,\\
&(\widetilde{\chi}^{mk})^{n}\,_{p}(\widetilde{\chi}^{ps})+
(\widetilde{\chi}^{mk})^{s}\,_{p}(\widetilde{\chi}^{np})+(\widetilde{\chi}^{mk})(\widetilde{\chi}^{ns})-(\widetilde{\chi}^{ns})(\widetilde{\chi}^{mk})=0.\label{eq3lie}
\end{split}
\end{align}
\begin{ex}
We consider the following four dimensional $3$-Lie algebra \cite{Bai2}
\begin{align*}
[e_2,e_3,e_4]=e_1,\qquad [e_1,e_3,e_4]=e_2,
\end{align*}
By solving the system of equations \eqref{eq3lie} we obtain the following $3$-Lie algebras as dual of $\mathcal{A}$
\begin{align*}
[\widetilde{e}^1,\widetilde{e}^2,\widetilde{e}^4]_\ast=b\widetilde{e}^1,\quad
[\widetilde{e}^3,\widetilde{e}^2,\widetilde{e}^4]_\ast=b\widetilde{e}^3,
\end{align*}
\begin{align*}
[\widetilde{e}^2,\widetilde{e}^3,\widetilde{e}^1]_\ast=b\widetilde{e}^2,\quad
[\widetilde{e}^4,\widetilde{e}^3,\widetilde{e}^1]_\ast=b\widetilde{e}^4,
\end{align*}
where $b$ is any non zero real number.
\end{ex}
\section{Conclusion}
In this paper, we defined the $3$-Leibniz and $3$-Lie bialgebras using cohomology of $3$-Leibniz and $3$-Lie algebras. Some theorems have been given , in particular, we have proven the correspondence between $3$-Leibniz bialgebra and its associated Leibniz bialgebras. There are some open problems related to this work. The definition of $r$-matrix and Yang-Baxter equation were related to $3$-Leibniz and $3$-Lie bialgebra. Applying the definition of $3$-Lie bialgebra in $M$ theory \cite{Bag1,Bag2,Gus} as a physical application is our future\cite{Aali}. We know that for the Nambu-Lie group $G$ \cite{Vai} on the dual space $\mathfrak{g}^\ast$ of the Lie algebra $\mathfrak{g}$ we have an $n$-Lie algebra structure. One can also investigate the concept of Nambu-Poisson-Lie group and the relation between $3$-Lie bialgebra and Lie bialgebra on the space $\mathfrak{g}$ \cite{Reza2}.
	
\subsection*{Acknowledgment}
\vspace{3mm}  We would like to express our deepest gratitude to M. Akbari-Moghanjoughi for carefully reading the manuscript and his useful comments and also we want to thank Sh. Ghanizadeh for giving impetus to this work.
This research was supported by Azarbaijan Shahid Madani University (Research Fund No. 27.d.1518).

\end{document}